\documentclass[12pt]{article}
\usepackage[english]{babel}
\usepackage{hyphenat}
\usepackage[latin2]{inputenc}
\usepackage{index}
\usepackage{graphicx}
\usepackage{amsfonts,amsmath,amssymb,amsthm,enumerate,latexsym}
\usepackage[includehead,includefoot,paper=a4paper,bindingoffset=1cm,top=2.5cm,bottom=2.5cm,left=2.5cm,right=2.5cm]{geometry}
\newcommand{\ds}{\displaystyle}
\let\subset\subseteq
\let\supset\supseteq
\def\cc#1{\overline{#1}}
\def\oV{\overline{V}}
\def\oW{\overline{W}}

\def\one{\overline{1}}
\def\Aplus{A^{+}}
\def\Bplus{B^{+}}
\def\Astar{A^{\star}}
\def\Bstar{B^\star}
\def\Aomega{A^{\omega}}
\def\Bomega{B^{\omega}}

\def\Finw{F^{-1}_{w_{[0,n)}}}
\def\er{\mathbb{R}}
\def\zet{\mathbb{Z}}
\def\en{\mathbb{N}}
\def\Er{\overline{\mathbb{R}}}
\def\ex{\mathbb{X}}
\def\ce{\mathbb{C}}
\def\Ce{\overline{\mathbb{C}}}
\def\We{{\mathcal W}}
\def\Ve{{\mathcal V}}

\def\Te{{\mathbb{T}}}
\def\De{{\mathbb{D}}}

\def\id{\mathop{\mathrm{id}}}

\def\Int{\mathop{\mathrm{Int}}}

\def\Tr{\mathop{\mathrm{Tr}}}
\def\el{{\mathcal{L}}}
\def\dd{{\mathrm{d}}}
\def\dot#1{{#1}^{\bullet}}
\let\epsilon\varepsilon
\def\picture#1#2#3{\begin{figure}
\begin{center}
\includegraphics{#2}
\caption{#1}
\def\tmp{#3}
\ifx\tmp\empty\else\label{#3}\fi
\end{center}
\end{figure}}
\def\zoomedpicture#1#2#3#4{\begin{figure}
\begin{center}
\includegraphics[width=#4]{#2}
\caption{#1}
\def\tmp{#3}
\ifx\tmp\empty\else\label{#3}\fi
\end{center}
\end{figure}}

\newtheorem{theorem}{Theorem}
\newtheorem*{thmCl}{Theorem \ref{thmClosedCover1}}
\newtheorem{lemma}[theorem]{Lemma}
\newtheorem{corollary}[theorem]{Corollary}
\newtheorem{observation}[theorem]{Observation}

\theoremstyle{definition}
\newtheorem{definition}[theorem]{Definition}
\newtheorem{example}[theorem]{Example}

\newenvironment{remark}{\bgroup\smallskip\noindent{\bf Remark.}\rm
}{\smallskip\egroup}

\author{Alexandr Kazda}
\title{Properties of Möbius number systems}
\begin{document}
\maketitle
\begin{abstract}
Möbius number systems represent points using sequences of Möbius
transformations. Thorough the paper, we are mainly interested in representing
the unit circle (which is equivalent to representing $\er\cup\{\infty\}$).

The main aim of the paper is to improve already known tools for proving that
a given subshift--iterative system pair is in fact a Möbius number
system. We also study the existence problem: How to describe iterative systems
resp. subshifts for which there exists a subshift resp. iterative system such
that the resulting pair forms a Möbius number system. While we were unable to
provide a complete answer to this question, we present both positive and
negative partial results. 

As Möbius number systems are also subshifts, we can ask when a given Möbius
number system is sofic. We give this problem a short treatment at the end of
our paper.


\end{abstract}

\section*{Introduction}
\addcontentsline{toc}{section}{Introduction}

Numeral systems are recipes for expressing numbers in symbols. The most common
are positional systems (usually with base ten, two, eight or sixteen). However, these
systems are by no means the only possibility. Various historical systems used
different approaches (consider for example the Roman numerals). The reason for
positional systems' eventual dominance is the ease with which we can perform
basic arithmetic operations on numbers by manipulating their symbolic
representations in a positional system (compared to the tedious task that is
arithmetic done in, say, Roman numerals).

Modern numeration theorists typically study positional systems with real base
(such as the golden mean or $-2$) or various modifications of continued
fractions. Contemporary numeration theory has connections to various other fields,
namely the study of fractals and tilings, symbolic dynamics, ergodic theory,
computability theory and even cryptography.

In this paper, we study M\"obius number systems as introduced in \cite{Kurka}.
A Möbius number system represents numbers as sequences of Möbius transformations
obtained by composing a finite starting set of Möbius transformations. 

Möbius number systems display complicated dynamical properties and have
connections to other kinds of numeral systems. In particular,
Möbius number systems can generalize continued fractions (see \cite{Kurka2}).

The paper is organized as follows: 
First, Section~\ref{chapPrel} prepares the ground for our study, introducing
known results on disc preserving Möbius transformations.

In Section~\ref{chapConv}, we present conditions for deciding whether a given
sequence of Möbius transformations represents a number. Originally,
representation was defined by convergence of measures, but it turns out that
there are several other definitions. We collect these definitions in
Section~\ref{secDefinitions}.

Section~\ref{chapMNS} introduces Möbius number systems and contains
most of this paper' results:
In Section~\ref{secIntervalSystems} we offer tools to prove (or disprove) that
a given iterative system with a given subshift is a Möbius number system. While
these tools are far from universal, they are sufficient for  practical
purposes, as we show on three example number systems. 

In Section~\ref{secQ}, we look at one particular tool, the numbers
$Q_n(\We,\Sigma)$, and ask what can these numbers tell us about Möbius number
systems. It turns out that if $Q_n(\We,\Sigma)$ are small enough then the
corresponding shift $\Sigma_\We$ either is not a Möbius number system or is
rather badly behaved.
 
An interesting question is whether a given iterative system admits any Möbius
number system at all. It would be quite useful to have a complete characterization of
such systems, unfortunately there is still a gap between the sufficient and the
necessary conditions that are available. In Section~\ref{secExistence} we offer
a slight improvement of an already known sufficient condition and observe a new
necessary condition. 

In Section~\ref{secData}, we present the findings of a computer experiment
based on the theory of Section~\ref{secExistence}. The results of this
experiment suggest that the set of Möbius iterative systems admitting a Möbius
number system is, up to a small error, equal to the complement of the set of
iterative systems having nontrivial inward set. We conjecture that this
is true in general, but can not offer a proof.

In Section~\ref{secSubshifts} we consider the question ``Which subshifts can be
Möbius number systems for a suitable Möbius iterative system?'' While we don't
know the full answer, we show that there is a nontrivial class of subshifts
that can not be Möbius number systems.

Proposition 5 in~\cite{Kurka2} offers a sufficient condition for a number
system to be a subshift of finite type. We take off in a similar direction in
Section~\ref{secSofic} and
conclude Section~\ref{chapMNS} by giving conditions for a Möbius number system
to be sofic. In particular, under some reasonable assumptions on the system
$\Sigma_\We$, we have a sufficient and necessary condition for $\Sigma_\We$ to
be sofic.

Finally, in the Appendix one can find various proofs of results belonging in
the folklore of the theory of Möbius transformations, symbolic dynamics or (in
one case) the theory of measure.

\section{Preliminaries}\label{chapPrel}

\subsection{Metric spaces and words}
Denote by $\Te$ the unit circle and by $\De$ the closed unit disc in the
complex plane. For $x,y\in \Te$, $x\neq y$ denote by $(x,y)$ resp. $[x,y]$ 
the open resp. closed interval obtained by
going from $x$ to $y$ along $\Te$ in the positive (counterclockwise) direction. To make notation more convenient, we define the sum $x+l$ for
$x\in \Te$ and $l\in\er$ as the point on $\Te$ whose argument is equal to $l+\arg x$ modulo $2\pi$.

Let $A$ be a finite alphabet. Any sequence of elements of $A$ is a \emph{word}
over $A$. Let $\lambda$ be the empty word. Denote by $\Astar$ the monoid of all
finite words over $A$, by $\Aplus$ the set $\Astar\setminus\{\lambda\}$ and by
$\Aomega$ the set of all one-sided infinite words over $A$. Let $|w|$ denote
the length of the word $w$. If $n$ is finite, let $A^n$ be the set of all words
over $A$ of length precisely $n$. We use the notation $w=w_0w_1w_2\cdots$ and
$w_{[i,j]}=w_iw_{i+1}\cdots w_{j}$. When $u$ is a finite word and $v$ any word
we can define the \emph{concatenation} of $u$ and $v$ as $uv=u_0u_1\dots
u_{|u|-1}v_0v_1\dots$

Let
$v\in\Astar$ be a word of length $n$. Then we write
$[v]=\{w\in\Aomega:w_{[0,n-1]}=v\}$ and call the resulting subset of $\Aomega$ the
\emph{cylinder} of $v$.
A word $u$ is a \emph{factor} of a word $v$ if there exist $i,j$ such that
$u=v_{[i,j]}$.

Let $X$ be a metric space. We denote by $\rho$ the metric function of
$X$, by $\Int(V)$ the interior of the set $V$ and by $B_r(x)$ the open ball of radius
$r$ centered at $x$. If $I$ is an interval, denote by $|I|$ the length of $I$.

We equip $\ce$ with the metric $\rho(x,y)=|x-y|$ and $\Te$ with the circle
distance metric (i.e. metric measuring distances along the circle). The shift
space $\Aomega$ of one-sided infinite words comes equipped with the metric
$\rho(u,v)=\max\left(\{2^{-k}:u_k\neq v_k\}\cup \{0\}\right)$. It is easy to
see that the topology of $\Aomega$ is the product topology. A \emph{subshift}
$\Sigma\subset\Aomega$ is a set that is both topologically closed and invariant
under the \emph{shift map} $\sigma(w)_{i}=w_{i+1}$ (i.e.
$\sigma(\Sigma)\subset\Sigma$).

As shown in \cite[pages 5 and 179]{Marcus}, subshifts of $\Aomega$ are precisely the subsets of
$\Aomega$ that can be defined by some set of forbidden factors. More precisely,
$\Sigma$ is a subshift iff there exists $F\subset \Aplus$ such that
\[
\Sigma=\{w\in\Aomega:\forall v\in F,\, \hbox{$v$ is not a factor of $w$}\}.
\]
We are going to occasionally define shifts using some such set of forbidden
factors.

The \emph{language of a subshift} $\el(\Sigma)$ is the set of all the words
$v\in\Astar$ for which there exists $w\in\Sigma$ such that $v$ is a factor of
$w$. See \cite{Marcus} for a more detailed treatment of this topic.
 
We will mainly consider symbolic representations of $\Te$, although 
representations of the 
extended real line $\Er=\er\cup \{\infty\}$ will make an appearance as
well. Note that $\Er$ 
is homeomorphic to $\Te$ via the stereographic projection (see
Figure~\ref{fig-stereo}):
\picture{The stereographic projection of $\Te$ onto
$\Er$}{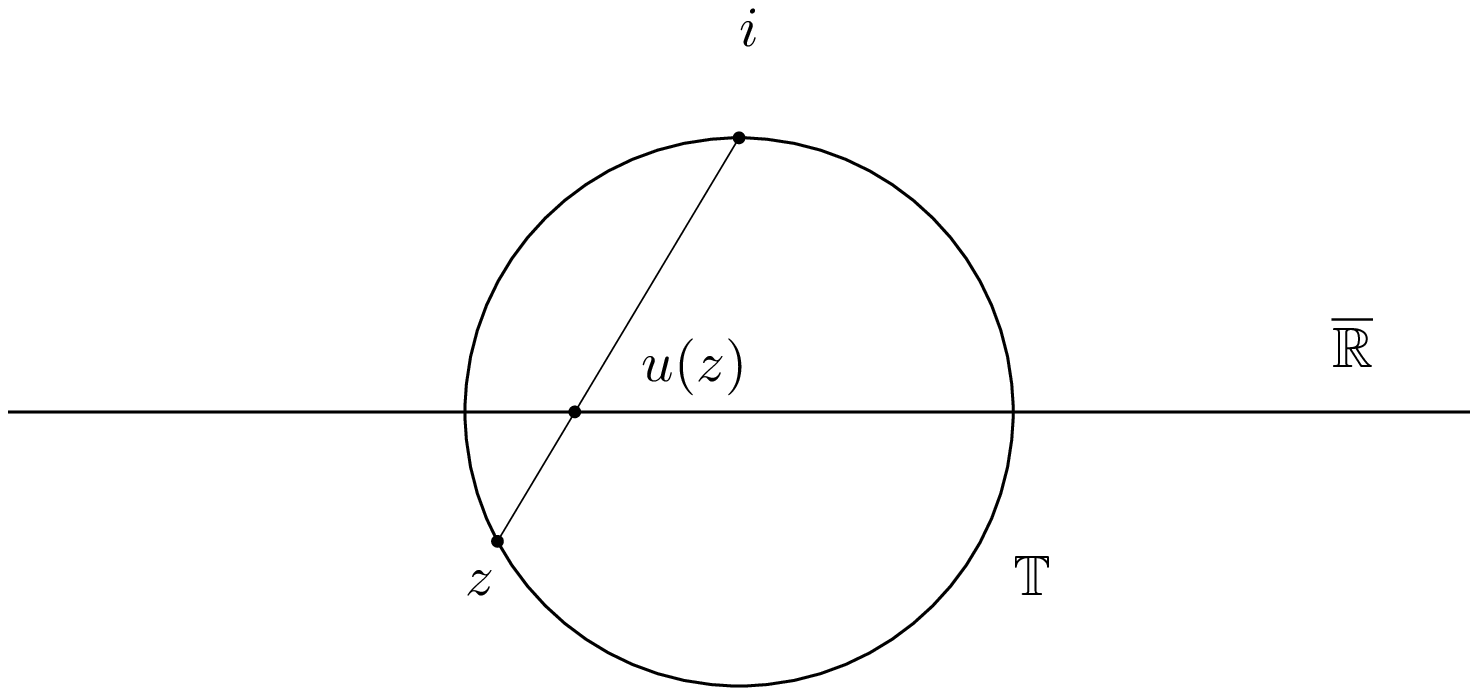}{fig-stereo}
$$
u:\Te \to \Er,\quad u:z \mapsto \frac{-iz+1}{z-i}.
$$
Therefore, as long as we are not interested in arithmetics,
representing $\Te$ is equivalent to representing the extended real line.

\subsection{Möbius transformations}

A Möbius transformation (MT for short) of the complex sphere
$\Ce=\ce\cup\{\infty\}$ is any map of the
form 
$$
F:z\mapsto\frac{az+b}{cz+d}
$$
where $(a,b),(c,d)$ are linearly independent vectors from $\ce^2$.

Note that the stereographic projection $u$ as defined above is 
actually a Möbius transformation. Therefore, if we represent $\Te$ using the
system $\{F_a:a\in A\}$ of MTs, we can represent $\Er$ in the same way with the
system $\{u\circ F_a\circ u^{-1}:a\in A\}$.

To every
regular $2\times 2$ complex matrix 
$
A=\left(\begin{smallmatrix}
a&b\\
c&d\\
\end{smallmatrix}\right)
$ we can associate the MT defined by $F_A(z)=\frac{az+b}{cz+d}$. While the map $A\mapsto
F_A$ is surjective, every MT has many preimages: if $A$ is a
matrix for $F$ then so is $c A$ for any $c\in\ce,\,c\neq 0$. Even
normalizing the matrices by demanding $\det A=1$ is not enough, as it leaves
two preimages $A$ and $-A$ for each $F$. 

This ambiguity is, however, a small price to pay: An easy calculation 
shows that composition of MTs corresponds to multiplying their respective matrices: 
$F_A\circ F_B=F_{A\cdot B}$. This is why we will often think of MTs as of
matrices. It
follows that the set of all MTs together with the operation of composition is a
group (isomorphic to $SL(2,\ce)/\{E,-E\}$). In particular, MTs are bijective on
$\ce\cup\{\infty\}$.

Usually, we will consider \emph{disc preserving} Möbius transformations, i.e.
transformations that map $\De$ onto itself.
Obviously, disc preserving transformations form a subgroup of the group of all MTs. It
turns out that $F$ is disc preserving iff it has the form
$$F=\begin{pmatrix}
\alpha&\beta\\
\overline\beta&\overline\alpha\\
\end{pmatrix}$$
with the normalizing condition $|\alpha|^2-|\beta|^2=1$. We will provide
the proof in the next section as Lemma~\ref{lemCPform}.

The geometrical theory of MTs is quite rich and has a strong link to
hyperbolic geometry (see \cite{Katok}). In this paper, we will need only a
handful of basic fragments of this theory. We will use the fact that MTs take
circles and lines to circles and lines (possibly turning a circle into a line
or vice versa) and the observation that disc preserving transformations also
preserve orientation of intervals on the circle (clockwise versus
counterclockwise), so the image of the interval $[x,y]$ is the interval
$[F(x),F(y)]$ (as compared to $[F(y),F(x)]$).

We can establish a taxonomy of disc preserving MTs by considering the trace
of the normalized matrix representing $F$. While $\Tr F$ does not have a well
defined sign, the number $(\Tr F)^2$ is unique and real for each disc preserving $F$.

\begin{definition}\label{defMT}
Let $F\neq \id$ be a disc preserving MT. We call $F$:
\begin{enumerate}
\item \emph{elliptic} if $(\Tr F)^2<4$,
\item \emph{parabolic} if $(\Tr F)^2=4$,
\item \emph{hyperbolic} if $(\Tr F)^2>4.$
\end{enumerate}
\end{definition}

To better understand this classification, consider the fixed points of $F$. We
claim (Lemma~\ref{lemFixedpoints} in the Appendix) that:
\begin{enumerate}
\item $F$ is elliptic iff it has one fixed point inside and one fixed point
outside of $\Te$ (the outside point might be $\infty$),
\item $F$ is parabolic iff it has a single fixed point which lies on $\Te$,
\item $F$ is hyperbolic iff it has two distinct fixed points, both lying on $\Te$.
\end{enumerate}

\begin{remark}
Let $F$ be a hyperbolic transformation with fixed points $x_1,x_2$. Then one of
these points (say, $x_1$) is \emph{stable} and the other is \emph{unstable}. It
is $F'(x_1)<1<F'(x_2)$ and for every $z\in\Ce, z\neq x_2$, we have 
$\lim_{n\to\infty} F^n(z)=x_1$.

Similarly, when $F$ is parabolic with the fixed point $x$, we have $F^n(z)\to x$
for all $z\in\Ce$ and $F'(x)=1$. See Lemmas~\ref{lemFixedPoints1} and
\ref{lemFixedPoints2} in the Appendix for proofs of these facts.
\end{remark}

We will show the significance of this classification in the following sections. 

\subsection{Number representation}

Möbius number systems assign numbers to sequences of mappings. This principle
is actually less exotic than it appears to be. Consider the usual binary
representation of the interval $[0,1]$. Let $A=\{0,1\}$ be our alphabet.
We want to assign to each word $w\in\Aomega$ the number
$\Phi(w)=0.w$ and so obtain the map $\Phi:\Aomega\to[0,1]$. We need to use
some sort of limit process: Taking longer and longer
prefixes of $w$, we obtain better and better approximations, ending with the
unique number $0.w$. 

The usual construction of the binary system involves letting 
$\Phi(w)$ to be equal to the limit of the sequence
$\{0.w_{[0,k)}\}_{k=1}^\infty$. However, we can also
define binary numbers in the language of mappings.

Consider the two maps 
\begin{eqnarray*}
F_0:x&\mapsto& x/2\\
F_1:x&\mapsto& (x+1)/2.
\end{eqnarray*}
For
$v\in A^n$ let $F_v=F_{v_0}\circ F_{v_1}\circ\dots\circ F_{v_{n-1}}$. Both maps
$F_0, F_1$ are continuous and, more importantly, contractions on the interval
$[0,1]$: For each $x,y\in[0,1]$ and each $i=0,1$ we have
$|F_i(x) - F_i(y)|=\frac12 |F_i(x)-F_i(y)|$. Therefore, for any $w\in\Aomega$,
the set $\bigcap_{k=1}^\infty F_{w_{[0,k)}}[0,1]$ is a singleton. What is more,
a proof by induction reveals that $F_{w_{[0,k)}}[0,1]$ is actually
precisely the set of all the real numbers whose binary expansion begins with
$0.w_{[0,k)}$. We have obtained that $\bigcap_{k=1}^\infty
F_{w_{[0,k)}}[0,1]=\{\Phi(w)\}=\{0.w\}$. If we wished, we could go on
to prove that $\Phi$ is continuous and surjective, both very desirable
properties for a number system.

We would like to do the same for Möbius transformations in place of $F_0,F_1$
and call the result a Möbius number system. However, as MTs are bijective on
the complex sphere, we cannot use the contraction property like we did above.
To fix this, \cite{Kurka} defined $\Phi$ using convergence of measures. We will
see that there are other (equivalent) definitions in
Theorem~\ref{thmCharacteriseConvergence} but let us give the original
definition first.

Denote $m(\Te)$ the set of all Borel probability measures on $\Te$.
If $\nu$ is a Borel
measure on $\Te$ and $F:\Te\to\Te$ an MT, we define the measure $F\nu$ by
$F\nu(E)=\nu(F^{-1}(E))$ for all measurable sets $E$ on $\Te$. The Dirac
measure centered at point $x$ is the measure $\delta_x$ such that 
\[
\delta_x(E)=
\begin{cases}
1&\text{if $x\in E$}\\
0&\text{otherwise}
\end{cases}
\]
for any $E$ measurable subset of $\Te$. It is a quite straightforward idea to identify
$\delta_x$ with the point $x$ itself.

Before we define what does it mean for a sequence of MTs to represent a point,
let us give some brief background. Denote by $C(\Te,\er)$ the vector space of all
continuous functions from $\Te$ to $\er$ (with the supremum norm). Finite Borel
measures act on $C(\Te,\er)$ as continuous linear functionals: Measure $\nu$
assigns to $f\in C(\Te,\er)$ the number $\int f \dd\nu$ and if $\nu\neq\nu'$
then the two measures define different functionals by the Riesz representation
theorem (see \cite[page 184]{ash}). 

We have the embedding $m(\Te)\subset C(\Te,\er)^*$, where
$C(\Te,\er)^*$ is the dual space to $C(\Te,\er)$. 
There are three usual topologies on $C(\Te,\er)^*$ (listed in the order of
strength): The norm
topology, the weak topology and the weak* topology. 
\begin{definition}\label{defRep}
Denote by $\mu$ the uniform probability measure on $\Te$. 
Let $\{F_n\}_{n=1}^\infty$ be a sequence of Möbius transformations. We say that
the sequence $\{F_n\}_{n=1}^\infty$ represents the point $x\in\Te$ if and only if
$\lim_{n\to\infty} F_n\mu=\delta_x$. Here $\mu$ is the uniform probability
measure on $\Te$ and the
convergence of measures is taken in the weak$^*$ topology, i.e. $\nu_n\to\nu$ if and only if
for all $f:\Te\to\er$ continuous we have $\int f\dd\nu_n\to \int f\dd\nu$.
\end{definition}

\begin{remark}
The reader might wonder why did we choose weak$^*$ topology here instead of any
the two other common topologies. 

One answer is that this is the usual way to
define convergence of measures in fields such as ergodic theory. Another answer
is that even weak topology is too strong to provide any representation of
points at all: Consider any sequence $\{F_n\}_{n=1}^\infty$ of MTs. To obtain
$\lim_{n\to\infty} F_n\mu=\delta_x$ in the weak topology, we would have to
satisfy $\alpha(F_n\mu)=\alpha(\delta_x)$ for any continuous linear functional
$\alpha\in C(\Te,\er)^{**}$.

By the Riesz representation theorem, the space $C(\Te,\er)^*$ can be identified
with the space of all Radon signed measures on $\Te$. For $\lambda$ Radon
signed measure on $\Te$, define $\alpha(\lambda)=\lambda(\{x\})$. This is
a continuous linear functional on $C(\Te,\er)^*$ (see Lemma~\ref{lemContinuity} in the Appendix).
 Obviously,
$\alpha(\delta_x)=1$, while $\alpha(F_n \mu)=\mu (\{F_n^{-1}(x)\})=0$, so the
sequence $\{F_n\}_{n=1}^\infty$ does not represent $x$. As this is true for all
sequences and all values of $x$, Definition~\ref{defRep} would be meaningless in
the weak topology and the same is true in the norm topology (which is even
stronger that the weak topology).
\end{remark}

\section{Representing points using Möbius transformations}
\label{chapConv}
\subsection{General properties of Möbius transformations}

In this subsection, we point out several useful properties of disc preserving MTs
as well as various equivalent descriptions of what does it mean for a sequence
of MTs to represent a point on $\Te$.

We begin by fulfilling a promise from Preliminaries:

\begin{lemma}\label{lemCPform}
A Möbius transformation $F$ is disc preserving (i.e. $F(\De)=\De$) iff it is of
the form 
\[
F=\begin{pmatrix}
\alpha&\beta\\
\cc\beta&\cc\alpha\\
\end{pmatrix},
\] where $|\alpha|^2-|\beta|^2=1$.
\end{lemma}
\begin{proof}
Let $F$ have the given form. We prove that then $F(\De)=\De$. First,
consider $z=e^{i\phi}$. We have:
\[
\left|\cc\beta e^{i\phi}+\cc\alpha\right|=
\left|(\cc{\alpha e^{i\phi}+\beta}) e^{i\phi}\right|=
\left|\alpha e^{i\phi}+\beta\right|.
\]
And so 
\[
|F(e^{i\phi})|=\frac{\left|\alpha e^{i\phi}+\beta\right|}{\left|\cc\beta
e^{i\phi}+\cc\alpha\right|}=1.
\]
Therefore $F(\Te)\subset \Te$. Because $F$ is an MT, the image of $\Te$ must be a
circle, so $F(\Te)=\Te$. The unit circle divides $\Ce$ into two components: The
inside (containing zero) and the outside (containing $\infty$). As $F$ is a 
bijection on $\Ce$, all we have to do to obtain $F(\De)=\De$ is prove
that $F(0)$ lies inside $\De$ . But this is simple:
$F(0)=\frac{\beta}{\cc\alpha}$ and $|\beta|<|\alpha|$, so $|F(0)|<1$.

On the other hand, consider any disc preserving MT
$F=\left(\begin{smallmatrix}
a&b\\
c&d\\
\end{smallmatrix}\right)$ where $\det F=1$. Because $F$ is continuous, it must be $F(\Te)=\Te$.
Therefore, for every $\phi$, we must have
$\left|a e^{i\phi}+b\right|= \left|c e^{i\phi}+d\right|.$
A little thought gives us that if $a=0$ then $d=0$ and similarly $b=0$ implies
$c=0$; in both cases we are done. Assume $a,b,c,d\neq 0$ and continue.

Choose $\phi$ so that the quantity 
$\left|a e^{i\phi}+b\right|= \left|c e^{i\phi}+d\right|$
is maximal. The maximal value of the
function on the left side is $|a|+|b|$, on the right side $|c|+|d|$, thus
$|a|+|b|=|c|+|d|$.
Similarly, by choosing the minimal quantity, we obtain that
$||a|-|b||=||c|-|d||$. Moreover, as $\phi$ is the same 
on the right and left, we also have $\arg a-\arg b = \arg c-\arg d$. These
three equalities will be enough to complete the proof.

Assume for a moment that the equality $||a|-|b||=||c|-|d||$ actually means
$|a|-|b|=|c|-|d|$. Then, together with $|a|+|b|=|c|+|d|$, we have
$|a|=|c|$ and $|b|=|d|$, obtaining a matrix of the form
$F=\left(\begin{smallmatrix}
a&b\\
ae^{i\psi}&be^{i\psi}\\
\end{smallmatrix}\right)$. But this matrix is singular, a contradiction.

Therefore, we must have $|a|-|b|=|d|-|c|$, which implies $|a|=|d|,|b|=|c|$ and,
after a brief calculation,
\[F=\begin{pmatrix}
a&b\\
\cc b e^{i\psi}&\cc ae^{i\psi}\\
\end{pmatrix}
\] for a suitable $\psi$. 

Now it remains to use the normalization formula $\det F=1$ to see 
that $\psi$ is either $0$ or $\pi$. If $\psi=0$, we are done. Otherwise, we
would have $\det F=|b|^2-|a|^2=1$, so $|b|>|a|$. But then
$|F(0)|=\frac{|b|}{|a|}>1$, so $F$ would turn the disc inside out, a
contradiction.
\end{proof}

\begin{definition}
Let $\alpha\in[0,2\pi),\,r\in[1,\infty)$. Call the transformation $R_\alpha(z)=e^{i\alpha}
z$ a \emph{rotation} and the transformation
$$
C_r=\frac1{{2}}
\begin{pmatrix}
r+\frac1r&r-\frac1r\\
r-\frac1r&r+\frac1r
\end{pmatrix},
$$
a \emph{contraction to 1}.
\end{definition}
Obviously, the identity map is both a contraction to $1$ and a rotation.
Moreover, rotations are precisely those disc preserving MTs whose matrices are diagonal.

\begin{remark}
Observe that any contraction to $1$ fixes the points $\pm 1$. The name
``contraction to 1'' comes from the fact that for $r>1$, the map $C_r$ is a
contraction in a suitable neighborhood of $1$ as can be seen by computing the
derivative $C_r'(1)=\frac1{r^2}$. Similarly, $C_r$ expands some neighborhood of
$-1$ as $C_r'(-1)=r^2$. Such $C_r$ is hyperbolic and acts on $\Te$ by making all points (with the exception of $-1$) ``flow'' towards $1$.  As we show in the next section, the sequence $\{C_n\}_{n=1}^\infty$ represents the point 1.
\end{remark}

\begin{lemma}\label{lemUniForm}
Let $F$ be a Möbius transformation. If $F$ is disc preserving then there exist
$\phi_1,\phi_2$ and $r$ such that $F=R_{\phi_1}\circ C_r \circ
R_{\phi_2}$. Moreover, if $F$ is not a rotation then $R_{\phi_1},R_{\phi_2}, C_r$ are uniquely determined by $F$.
\end{lemma}
\begin{proof}
We want to satisfy the equation
$$\begin{pmatrix}
\alpha&\beta\\
\cc\beta&\cc\alpha\\
\end{pmatrix}=F
=R_{\phi_1}\circ C_r \circ R_{\phi_2}=
\begin{pmatrix}
\frac12\left(r+\frac1r\right)e^{i\frac{\phi_1+\phi_2}2}&
\frac12
\left(r-\frac1r\right)e^{i\frac{\phi_1-\phi_2}2}\\
\frac12\left(r-\frac1r\right)e^{-i\frac{\phi_1-\phi_2}2}& 
\frac12\left(r+\frac1r\right)e^{-i\frac{\phi_1+\phi_2}2}\\
\end{pmatrix}.
$$
If $\beta=0$ then $F$ is a rotation and there are many solutions to the above
equation; for example
$r=1,\,\phi_1=2\arg\alpha,\,\phi_2=0$. Let us now assume $\beta\neq 0$.

Choose $r$ so that $\frac1{2}\left(r+\frac1r\right)=|\alpha|$.
It is easy to see that we will then have
$\frac1{2}\left(r-\frac1r\right)=|\beta|$ so it remains to get the arguments of
$\alpha$ and $\beta$ right. 

Obviously, we need to choose the parameters of the rotations $\phi_1$ and
$\phi_2$ so that 
we satisfy the conditions $\phi_1+\phi_2=2\arg \alpha$ and
$\phi_1-\phi_2=2\arg\beta$. But this is a linear system with a single solution,
therefore $\phi_1,\phi_2$ are unique (modulo $2\pi$, of course).
\end{proof}

Denote by $\dot F (x)$ the modulus of the derivative of $F$
at $x$. Direct calculation gives us that when
$F=\left(\begin{smallmatrix}\alpha&\beta\\\cc\beta&\cc\alpha\end{smallmatrix}\right)$,
$\det F=1$ then
\[
\dot F (x)=|F'(x)|=\frac1{\left|\cc \beta x+\cc \alpha\right|^2}.
\]
This number measures whether and how much $F$ expands or contracts the neighborhood of $x$.

\begin{definition}
Let $F$ be a Möbius transformation. Then, inspired by \cite{Kurka} and
\cite{Katok}, we define the four sets
\begin{align*}
U&=\{x\in\Te: \dot{F}(x)<1\}\\
V&=\{x\in\Te: \dot{(F^{-1})}(x)>1\}\\
C&=\{x\in\overline{\ce}: \dot{F}(x)\geq1\}\\
D&=\{x\in\overline{\ce}: \dot{(F^{-1})}(x)\geq1\}.\\
\end{align*}
Call $U$ the \emph{contraction interval of $F$}, $V$ the \emph{expansion
interval of $F^{-1}$} and $C$ resp. $D$ the \emph{expansion sets} of $F$
resp. $F^{-1}$.
\end{definition}

By Lemma~\ref{lemUniForm} we have that for every $F$ there exist
$\phi_1,\phi_2$ and $r$ such that $F=R_{\phi_1}\circ C_r\circ R_{\phi_2}$. As
$\dot R_\phi=1$, we have $\dot F(x)=\dot{C_r} (R_{\phi_2}(x))$ and
$\dot{(F^{-1})}(x)=\dot {(C_{r}^{-1})}(R_{-\phi_1}(x))$. Because the sets $U$ and $C$
are defined using $\dot F(x)=\dot{C_r} (R_{\phi_2}(x))$, the value of $r$ determines the
shapes and sizes of $U$ and $C$ while $\phi_2$ rotates $U$ and $C$ clockwise
around the point $0$. Similarly, the shapes of $V$ and $D$ depend on $r$ while
$\phi_1$ determines positions of $V$ and $D$, rotating them (counterclockwise)
around $0$.

\begin{lemma}\label{lemMobiology}
Let $F=\left(\begin{smallmatrix}\alpha&\beta\\\cc\beta&\cc\alpha\end{smallmatrix}\right)$ be a Möbius transformation that is not a rotation. 
Then the following holds:
\begin{enumerate}
\item $C$ and $D$ are circles with the same radius $|\beta|^{-1}$ and centers
$c,d$ such that $c=-\frac{\cc\alpha}{\cc\beta}$, $d=\frac{\alpha}{\cc\beta}$.
Moreover, $|c|=|d|=\sqrt{|\beta|^{-2}+1}$.
\item $U=\Te\setminus C$ and $V=\Te\cap \Int(D)$
\item $F(\Ce\setminus C)=\Int D$
\item $F(U)=V$
\item $|V|<\pi$
\item $|U|+|V|=2\pi$
\item If $x\neq y$ are points in $V$ then $I$, the shorter of the
two intervals joining $x,y$, lies in $V$.
\end{enumerate}
\end{lemma}
\begin{proof}
We prove (1) by direct calculation. We have $\dot F (x)=\frac1{\left|\cc \beta x+\cc
\alpha\right|^2}$ and therefore $x\in C$ if and only if 
$$\left|x+\frac{\cc\alpha}{\cc\beta}\right|\leq|\beta|^{-1}.$$
This is the equation of a disc with the center
$c=-\frac{\cc\alpha}{\cc\beta}$ and radius $|\beta|^{-1}$. Also, it is
$$|c|=\frac{|\alpha|}{|\beta|}=\frac{\sqrt{1+|\beta|^2}}{|\beta|}=\sqrt{|\beta|^{-2}+1}.$$

The case of $F^{-1}$ is similar.

Observe that $\dot{(F^{-1})}(z)=1$ precisely
on the boundary of $D$. This (together with $\Te=(\Te\cap C)\cup U$) gives us (2).

Parts (3) and (4) follow from the formula for the derivative of a composite
function. We can write $\dot{F^{-1}}(F(z))=\frac{1}{\dot{F}(z)}$ and so $\dot F(z)<1$ if and only if
$\dot{F^{-1}}(F(z))>1$. 

\picture{The geometry of $C$ and $D$}{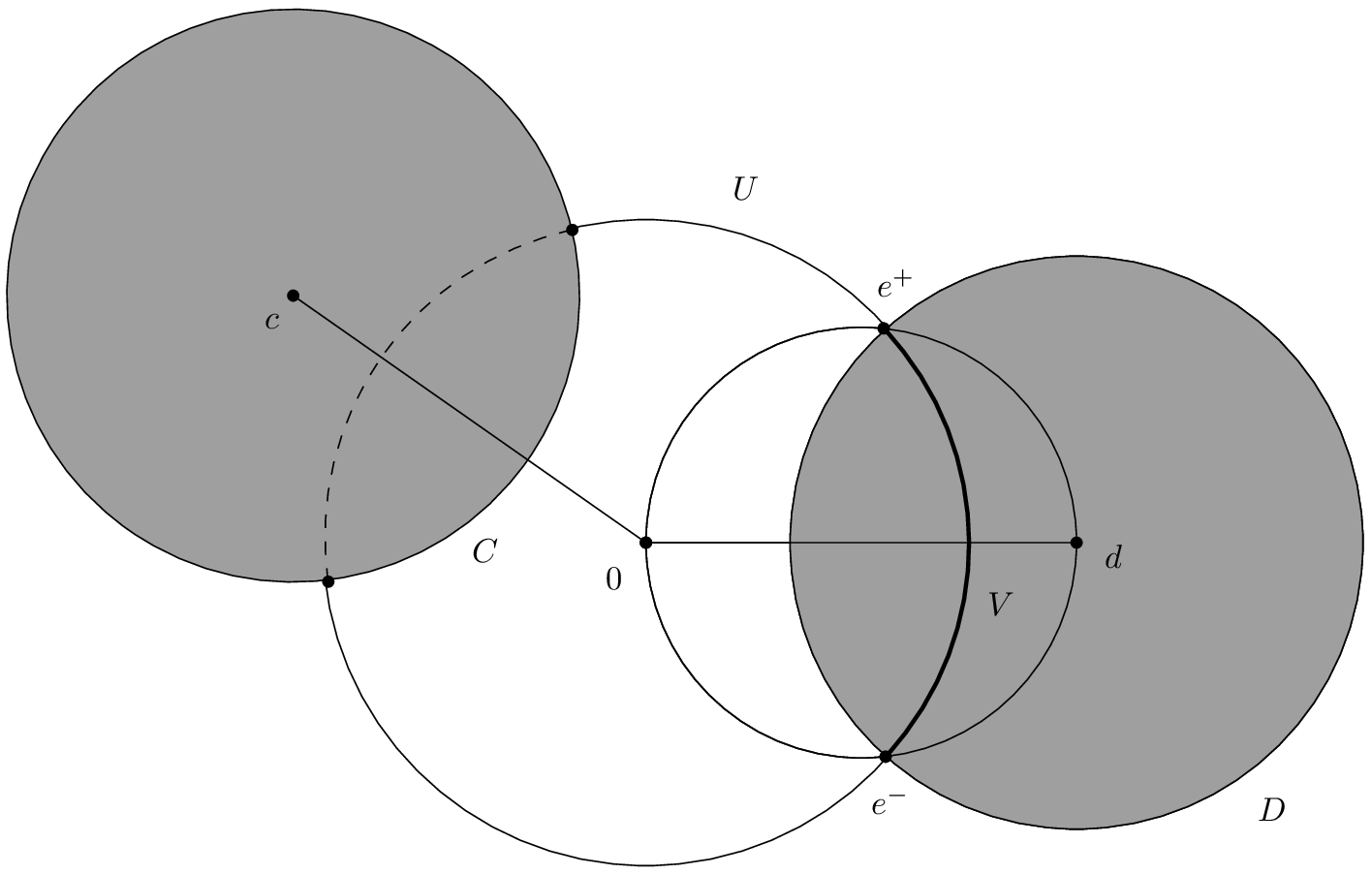}{figCDgeometry}
Elementary geometrical analysis of the situation yields (5) and (6)
(see Figure~\ref{figCDgeometry}). Finally, (7) is a direct consequence of (5).
\end{proof}
\begin{remark}
Observe that the triangles $0de^+$ and $0de^-$ in Figure~\ref{figCDgeometry} 
are right by Pythagoras' theorem. Also, we can compute that the length of $V$ is equal to 
$2\arccos\left(\frac{|\beta|}{\sqrt{1+|\beta|^2}}\right)$ and the distance of $d$ from $V$
is $\sqrt{1+|\beta|^{-2}}-1$. Therefore, the size of $D$, length of $V$ and the
distance of $d$ and $\Te$ are all decreasing functions of $|\beta|$. This will
be important in Theorem~\ref{thmCharacteriseConvergence}.
\end{remark}

\subsection{Representing $\Te$ and $\Er$}

Recall that a sequence of Möbius transformations $\{F_n\}_{n=1}^\infty$ 
represents the point $x\in\Te$ if and only if
$\lim_{n\to\infty} F_n\mu=\delta_x$ in the weak* topology. 
We will now list several equivalent definitions of what does it mean to
represent a point. Some of these results were already known (see Proposition 3
in \cite{Kurka2}).

\label{secDefinitions}
\begin{theorem}
\label{thmCharacteriseConvergence}
Let $\{F_n\}_{n=1}^\infty$ be a sequence of MTs only finitely many of which are rotations. 
Denote by $V_n$ the expansion interval of
$F^{-1}_n$, by $D_n$ the expansion set of $F_n^{-1}$ and by $d_n$ the center 
of $D_n$. Then the following statements are equivalent:
\begin{enumerate}
\item The sequence $\{F_n\}_{n=1}^\infty$ represents $x\in \Te$.
\item For every open interval $I$ on $\Te$ containing $x$ we have
$\lim_{n\to\infty}(F_n\mu)(I)=1$.
\item There exists a number $c>0$ such that for every open interval 
$I$ on $\Te$ containing $x$ it is true that
$\lim\inf_{n\to\infty}(F_n\mu)(I)>c$.
\item  $\ds\lim_{n\to\infty}d_n= x$
\item  $\ds\lim_{n\to\infty}D_n= \{x\}$
\item  $\ds\lim_{n\to\infty}\overline{V}_n=\{x\}$
\item  For all $K\subset \Int(\De)$ compact we have $\ds\lim_{n\to\infty}F_n(K)= \{x\}$.
\item  For all $z\in\Int(\De)$ we have $\ds\lim_{n\to\infty}F_n(z)= x$ .
\item  There exists $z\in\Int(\De)$ such that $\ds\lim_{n\to\infty}F_n(z)= x$.
\item The sequence $\{F_n\}_{n=1}^\infty$ converges to the constant map $c_x: z \mapsto x$ in measure,
that is
\[\forall \epsilon>0,\,
\lim_{n\to\infty}\mu(\{z:\rho(F_n(z),x)>\epsilon\})=0.\]
\end{enumerate}

Here, $\mu$ is the uniform probability measure on $\Te$ and $\rho$ the metric
on $\Te$.
In (5), (6) and (7), we take convergence in the Hausdorff metric on the
space of nonempty compact subsets of $\ce$, $\Te$ and $\De$ respectively. In particular,
$E_n\to\{x\}$ if and only if for every $\epsilon>0$ there exists $n_0$ such that $\forall
n>n_0$ it is $E_n\subset B_\epsilon(x)$.
\end{theorem}
\begin{proof}
We prove a sequence of implications. Unfortunately, the easiest to understand
sequence of implications that we found is a bit more complicated than the usual
``wheel'' used to prove theorems of this type. See Figure~\ref{figNonwheel} for
our global plan.
 
\picture{The sequence of implications used to prove
Theorem~\ref{thmCharacteriseConvergence}}{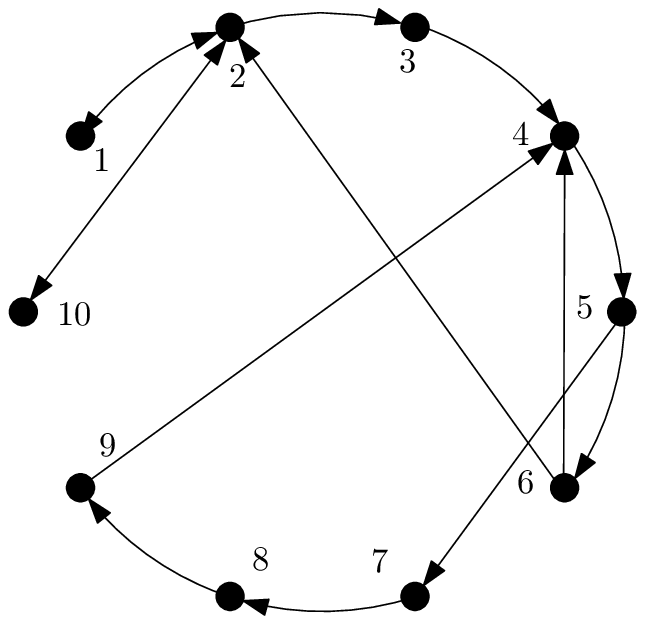}{figNonwheel}

Assume (1). Let $I=(a,b)$ be an open interval containing $x$.
Consider the function $f$ defined by:
\[
f(z)=
\begin{cases}
1&\text{if $z\in [b,a]$}\\
\frac{\rho(x,z)}{\rho(x,a)}&\text{if $z\in [a,x]$}\\
\frac{\rho(x,z)}{\rho(x,b)}&\text{if $z\in [x,b]$}
\end{cases}
\]
Obviously, $f$ is continuous on $\Te$ and $f(x)=0$ (see Figure~\ref{fig-ridge}
for the situation in the case $|I|<\pi$).
By definition, it is $\int f \dd F_n\mu\to f(x)=0$. Now consider that
$f(I^c)=1$ and so $\int f \dd F_n\mu\geq F_n\mu(I^c)$. This means 
$F_n\mu(I^c)\to 0$ and so $F_n\mu(I)\to 1$, proving (2).

\picture{The graph of $f$ used to prove
$(1)\Rightarrow(2)$}{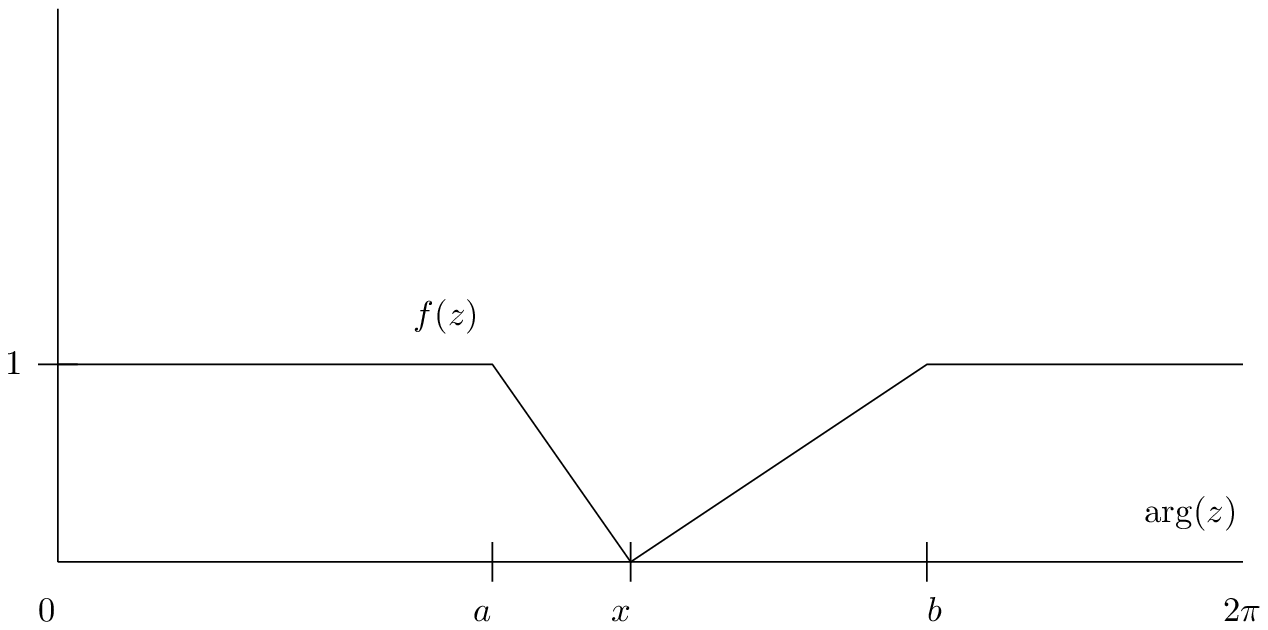}{fig-ridge}

To prove $(2)\Rightarrow(1)$, consider any continuous function $f:\Te\to\er$.
As $\Te$ is compact, $f$ is bounded by some $M$. 
For every $\epsilon>0$ there exists $\delta>0$ such that whenever 
$y\in I=(x-\delta,x+\delta)$, it is $|f(x)-f(y)|<\epsilon$. By (2), there
exists $n_0$ such that whenever $n>n_0$, we have $F_n\mu (I)>1-\epsilon$.
Therefore (assuming without loss of generality $\epsilon<1$):
\[
\left|f(x)-\int f \dd F_n\mu\right|\leq \epsilon(1-\epsilon)+2M\epsilon<(2M+1)\epsilon.
\]
This proves $\lim_{n\to\infty}\int f \dd F_n\mu=f(x)=\int f\dd\delta_x$,
verifying weak$^*$ convergence.

Claim (3) easily follows from (2) by setting $c=\frac12$.

Assume that (3) is true. Denote
$F_n=\begin{pmatrix}
\alpha_n&\beta_n\\
\cc\beta_n&\cc\alpha_n
\end{pmatrix}
$.

Let $\epsilon>0$ and take the interval
$I=(x-c\epsilon,x+c\epsilon)$. There exists $n_0$ such that for all $n>n_0$ 
we have $|F^{-1}_n(I)|>\frac c2$. Therefore, for some $z\in I$ the
inequality
$\dot{(F^{-1}_n)}(z)>\frac{c}{4c\epsilon}= \frac1{4\epsilon}$ holds. 

Recall that $\dot{(F^{-1}_n)}(z)=\frac1{|-\cc\beta_n z+\alpha_n|^2}$.
 Moreover, for $|z|=1$ we have
$|-\cc\beta_n z+\alpha_n|\geq |\alpha_n|-|\beta_n|$, so
$\dot{(F^{-1}_n)}(z)\leq (|\alpha_n|-|\beta_n|)^{-2}$.
Using $|\alpha_n|^2-|\beta_n|^2=1$, we
obtain
\[
\frac1{4\epsilon}<\dot{(F^{-1}_n)}(z)\leq (|\alpha_n|+|\beta_n|)^{2},
\]
so that for small enough $\epsilon$, we have
$|\beta_n|>1$ for all $n>n_0$. 

Now from $\frac1{|-\cc\beta_n z+\alpha_n|^2}>\frac{1}{4\epsilon}$ 
we obtain
$|\frac{\alpha_n}{\cc\beta_n}-
z|<\frac{2\sqrt\epsilon}{\cc\beta_n}<2\sqrt\epsilon$. 
Recall that $\frac{\alpha_n}{\cc\beta_n}$ is precisely the point $d_n$. We have
just shown that for any small enough $\epsilon>0$ there is $n_0$ such that for
all $n>n_0$ we can find $z$ such that:
\[
|x-d_n|<|x-z|+|z-d_n|<c\epsilon+2\sqrt\epsilon
\]
implying $d_n\to x$.

An elementary examination of the geometry of $V_n$ and $D_n$ shows that
(4), (5) and (6) are all equivalent. In particular, if $d_n\to x$ then the
diameter of $D_n$ tends to zero and so $D_n\to\{x\}$. Moreover,
$\oV_n=\Te\cap D_n$, so if $D_n\to\{x\}$ then $\oV_n\to\{x\}$ as well. It remains
to see that if $\oV_n\to\{x\}$ then $|V_n|\to 0$ which can only happen when
$|\beta_n|\to\infty$. Therefore $\rho(V_n,d_n)=\rho(\Te,d_n)$ tends to
zero, meaning that $d_n\to x$.

To prove $(6)\Rightarrow(2)$, consider any open interval $I\subset \Te$ containing
$x$. We know that $\oV_n\to\{x\}$, so there exists $n_0$ such that $n>n_0\Rightarrow
\oV_n\subset I$. Furthermore, for all $\epsilon>0$, we can find $n_\epsilon>n_0$ such
that $|V_n|<\epsilon$ whenever $n>n_\epsilon$. For $n>n_\epsilon$, we now have
the chain of inequalities:
\[
|F^{-1}_n (I)|\geq |F^{-1}_n(V_n)|=2\pi-|V_n|>2\pi-\epsilon,
\]
where the middle equality comes from Lemma~\ref{lemMobiology}. We have proved
$\mu (F^{-1}_n(I))\to 1$.

Denote now by $C_n$ the expansion set of $F^{-1}_n$ and assume (5). Observe 
that the diameter of $C_n$ is equal to the diameter of $D_n$ and so the
diameter of $C_n$ tends to 0. As $C_n$ is a circle with center outside
$\De$, for any $K\subset\Int(\De)$ compact there
exists $n_0$ such that $K\cap C_n=\emptyset$ whenever $n>n_0$. 
Then $F_n(K)\subset D_n$ and so $F_n(K)\to\{x\}$, proving (7).

As $\{z\}$ is a compact set, (8) easily follows from (7). Also, statement (9) is an
obvious consequence of (8).

We now prove $(9) \Rightarrow (4)$. First consider the case $z=0$. Then
$F_n(0)\to x$ iff $\frac{\beta_n}{\cc\alpha_n}\to x$. However,
$d_n=\frac{\alpha_n}{\cc\beta_n}$, therefore $d_n\to \frac 1{\cc x}$. As
$x\in\Te$ and the map
$z\mapsto \frac1{\cc z}$ is the circle inversion with respect to $\Te$ , 
we have $\frac1{\cc x}=x$, so $d_n\to x$.

In the case $z\neq 0$, we will make use of the already known equality
$(2)\Leftrightarrow (4)$.

Let $F_n(z)\to x\in \Te$ with $z\neq 0$. Let $M$ be any disc preserving MT
$M$ that sends $0$ to $z$ (it is easy to find such an MT, as $M(0)=z$ iff
$\frac{\beta}{\cc\alpha}=z$). 

Let $G_n=F_n\circ M$ and observe that $G_n(0)=F_n(z)\to x$. Therefore, as we
have just shown, the sequence $\{G_n\}_{n=1}^\infty$
satisfies (4) and (by $(4)\Rightarrow (2)$) we have $\mu(M^{-1}(F_n^{-1}(I)))\to 1$
whenever $I$ is an open interval containing $x$. 

But $M(\mu)$ is absolutely
continuous with respect to $\mu$ and so 
\[
\mu(M^{-1}(F_n^{-1}(I^c)))\to 0 \Rightarrow \mu(F_n^{-1}(I^c))\to 0.
\]
We then have $\mu(F_n^{-1}(I))\to 1$, therefore the sequence
$\{F_n\}_{n=1}^\infty$ satisfies (2). Using $(2)\Rightarrow(4)$, we finally obtain $(9)\Rightarrow (4)$.

It remains to show $(2) \Leftrightarrow (10)$, which turns out to be a simple
exercise: Assume $(2)$. Then $\forall \epsilon>0$ we have
\[
\mu(F_n^{-1}(x-\epsilon,x+\epsilon))\to 1\Rightarrow
\mu(F_n^{-1}(\{z:\rho(z,x)>\epsilon\}))\to 0.
\]
But
$F_n^{-1}(\{z:\rho(z,x)>\epsilon\}=\{z:\rho(F_n(z),x)>\epsilon\}$, proving
convergence in measure.

Similarly, if $\{F_n\}_{n=1}^\infty$ converges in measure, we obtain
that $\mu(F_n^{-1}(x-\epsilon,x+\epsilon))$ tends to 1 for each $\epsilon>0$.
Obviously, every open $I$ such that $x\in I$ contains an interval of the form
$(x-\epsilon,x+\epsilon)$, proving (2).
\end{proof}
\begin{remark}
Note that Theorem~\ref{thmCharacteriseConvergence} is mostly true even if there
are infinitely many rotations in the sequence $\{F_n\}_{n=1}^\infty$. In this
case, all the statements are trivially false, with the exception of (4) and (6)
that are not well defined (though (4) can be easily fixed by letting
$d_n=\infty$ for $F_n$ rotation).
\end{remark}

As an easy corollary of Theorem~\ref{thmCharacteriseConvergence}, we can 
prove that two intuitive ideas are true.
\begin{corollary}\label{corCompose}
Let $\{F_n\}_{n=1}^\infty$ be a sequence of MTs representing the point $x$. Let $M$
be a disc preserving MT. Then
\begin{enumerate}
\item The sequence $\{F_n\circ M\}_{n=1}^\infty$ represents $x$.
\item The sequence $\{M\circ F_n\}_{n=1}^\infty$ represents $M(x)$.
\end{enumerate}
\end{corollary}
\begin{proof}
In both cases, we use the fact that if $G_n(z)\to x$ for some $z\in\Int(\De)$
then the sequence $\{G_n\}_{n=1}^\infty$ represents $x$.
\begin{enumerate}
\item As $M(0)$ lies inside $\De$, we have $F_n(M(0))\to x$, therefore
$(F_n\circ M)(0)\to x$.
\item As $F_n(0)\to x$ and $M$ is continuous, we have $M(F_n(0))\to
M(x)$.\qedhere
\end{enumerate}
\end{proof}
We now have enough tools to show, like in~\cite{Kurka}, how do the three classes of MTs behave with respect to point representation:
\begin{enumerate}
\item Let $F$ be an elliptic disc preserving transformation. Then the sequence
$\{F^n\}_{n=1}^\infty$ does not represent any point.
\item Let $F$ be a parabolic disc preserving transformation. Then the sequence
$\{F^n\}_{n=1}^\infty$ represents the fixed point of $F$.
\item Let $F$ be a hyperbolic disc preserving transformation. Then the sequence
$\{F^n\}_{n=1}^\infty$ represents the stable fixed point of $F$.
\end{enumerate}

For all three claims, we will need part (8)
of Theorem~\ref{thmCharacteriseConvergence}.

To prove (1), recall that if $F$ is elliptic, there exists a fixed point of $F$
inside $\Te$. Denote this point by $x$. Then for all $n,\,F^n(x)=x\not\in\Te$,
so $\{F^n\}_{n=1}^\infty$ can not represent anything.

In the parabolic and hyperbolic case, denote by $x$ the (stable) fixed point of
$F$. We now use Lemma~\ref{lemFixedPoints2}
in the Appendix to obtain that for all $z\in\Int(\De)$ we have $F^n(z)\to x$.
Therefore, $\{F^n\}_{n=1}^\infty$ represents $x$, proving (2) and (3).

Going in a different direction, we obtain a useful sufficient condition
for representing a point.
\begin{corollary}
\label{corConvergencyTools}
Let $\{F_n\}_{n=1}^\infty$ be a sequence of Möbius transformations such that
for some $x_0\in\Te$ we have $\lim_{n\to\infty}\dot{(F^{-1}_n)}(x_0)=\infty$. Then 
$\{F_n\}_{n=1}^\infty$ represents $x_0$.
\end{corollary}
\begin{proof}
Let
$F_n=\left(\begin{smallmatrix}\alpha_n&\beta_n\\\cc\beta_n&\cc\alpha_n\end{smallmatrix}\right)$,
$|\alpha_n|^2-|\beta_n|^2=1$.
Obviously, we have $|-\cc\beta_n x_0+\alpha_n|\to 0$. Because $|-\cc\beta_n
x_0+\alpha_n|\geq |\alpha_n|-|\beta_n|$, we obtain that 
$|\beta_n|\geq 1$ for all $n$ large enough.
Then from 
$$
\dot{(F^{-1}_n)}(x_0)=\frac1{|-\cc\beta_n x_0+\alpha_n|^2}
$$
we have $\frac{\alpha_n}{\cc\beta_n}\to x_0$. But
$\frac{\alpha_n}{\cc\beta_n}=d_n$, so $d_n\to x_0$. Now
 $\{F_n\}_{n=1}^\infty$ must represent $x_0$ by part (4) of Theorem
\ref{thmCharacteriseConvergence}.
\end{proof}

Note that the sequence $\{F^n\}_{n=1}^\infty$ with $F$ parabolic is a
counterexample to the converse of Corollary~\ref{corConvergencyTools}. This
sequence represents the fixed point $x$ of $F$, yet $\dot {(F^{-n})}(x)=1$ for
all $n$.

\begin{remark}
To show that the part (8) of Theorem~\ref{thmCharacteriseConvergence} 
can not be improved to include points on $\Te$, we give an 
example of a sequence $\{F_n\}_{n=1}^\infty$ of
MTs such that $\{F_n\}_{n=1}^\infty$  represents 
the point $1$ while the $\{F_n(z)\}_{n=1}^\infty$ does not converge to $1$.

Given the contraction $C_n$, denote $E_n=C_n^{-1}([i,-i])$. The set $E_n$ is an
interval such that $z\in E_n\Rightarrow \rho(C_n(z),1)>\frac{\pi}2$ (see
Figure~\ref{figEn}). 
\picture{The set $E_n$.}{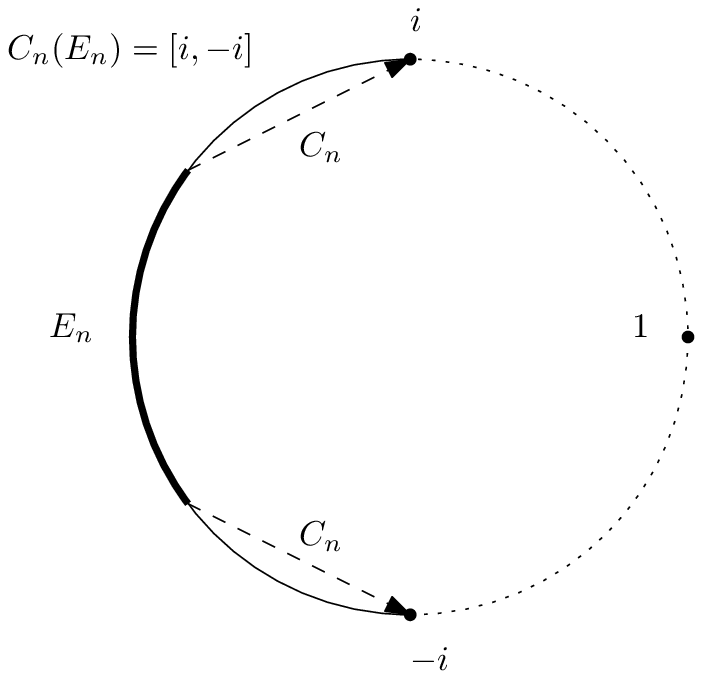}{figEn}
For
$n\in\en$, let $\alpha_{n,1},\alpha_{n,2},\dots,\alpha_{n,m_n}$ be angles of
rotation such
that $\bigcup_{i=1}^{m_n} R^{-1}_{\alpha_{n,i}}(E_n)=\Te$. It remains to
consider the sequence of transformations 
\[
C_1\circ R_{\alpha_{1,1}},C_1\circ R_{\alpha_{1,2}},\dots,C_1\circ
R_{\alpha_{1,m_1}},C_2\circ R_{\alpha_{2,1}},\dots,C_2\circ
R_{\alpha_{2,m_2}},\dots
\]
This sequence represents the point $1$ by Corollary~\ref{corConvergencyTools}
(although the speed of convergence is quite low). Moreover, for all $z\in\Te$ and all $n$ there exists $i$ such that 
$z\in R^{-1}_{\alpha_{n,i}}(E_n)=(C_n\circ R_{\alpha_{n,i}})^{-1}([i,-i]),$
so images of $z$ do not converge to 1.

\end{remark}

In contrast to the above construction, we can always achieve  pointwise
convergence almost everywhere by taking subsequences:

\begin{corollary}
If $\{F_n\}_{n=1}^\infty$ represents $x$ then there exists a subsequence
$\{F_{n_k}\}_{k=1}^\infty$ such
that $F_{n_k}(z)\to x$ almost everywhere for $k\to\infty$.
\end{corollary}
\begin{proof}
As the sequence $\{F_n\}_{n=1}^\infty$ represents $x$, it converges in measure
to the constant function $c_x(z)=x$. Applying the Riesz theorem (different from
the Riesz representation theorem) from
\cite[page 47] {malylukes}, we obtain that there exists a subsequence
$\{F_{n_k}\}_{k=1}^\infty$ such
that $F_{n_k}(z)\to c_x(z)=x$ for almost all $z$.
\end{proof}
\begin{remark}
When representing the real line, analogous results hold. Instead of the
interior of $\De$ we have the upper half plane $\{z\in\ce:\Im(z)> 0\}$ and
instead of $\mu$ we can take either the image of $\mu$ under the stereographic
projection, or any other Borel probabilistic measure that is absolutely
continuous with respect to the Lebesgue measure on $\er$.
\end{remark}

\section{Möbius number systems}\label{chapMNS}
\subsection{Basic definitions and examples}\label{secBasic}
Let $A$ be an alphabet. 
Assume we assign to every $a\in A$ a Möbius transformation $F_a$. The
set $\{F_a:a\in A\}$ is then called a \emph{Möbius iterative system}. Given an
iterative system, we 
assign to each word $v\in A^n$ the mapping $F_v=F_{v_0}\circ
F_{v_1}\circ\cdots\circ F_{v_{n-1}}$. 

\begin{definition}
Given $w\in\Aomega$, we define $\Phi(w)$ as the point $x\in\Te$ such that
the sequence $\{F_{w_{[0,n)}}\}_{n=1}^\infty$ represents $x$. If
$\{F_{w_{[0,n)}}\}_{n=1}^\infty$ does not represent
any point in $\Te$, let $\Phi(w)$ be undefined. Denote the 
domain of the resulting map $\Phi$ by $\ex_F$.
\end{definition}

\begin{definition}
The subshift $\Sigma\subset\Aomega$ is a \emph{Möbius number system} for a
given Möbius iterative system if
$\Sigma\subset \ex_F$, $\Phi(\Sigma)=\Te$ and $\Phi_{|\Sigma}$ is
continuous.
\end{definition}

Using Corollary~\ref{corCompose}, we observe that if $\Phi(w)=x$ then
$\Phi(\sigma(w))=F^{-1}_{w_0}(x)$. We will use this simple property later.

We now give three examples of Möbius number systems, although the proof that
they indeed are Möbius number systems will have to wait until
Section~\ref{secExamples2} when we have suitable tools.

\begin{example}\label{exp3parabolic}
Let $A,B,C$ be three 
vertices of an equilateral triangle inscribed in $\Te$.
Take $F_a,F_b,F_c$ the three parabolic transformations satisfying.
\begin{eqnarray*}
F_a(A)=A,& F_a(C)=B\\
F_b(B)=B,& F_b(A)=C\\
F_c(C)=C,& F_c(B)=A.\\
\end{eqnarray*}
See Figure~\ref{fig3parabolic}. A quick calculation reveals that $F_a,F_b,F_c$
are in fact uniquely determined by the triangle $ABC$.

Let us define the shift $\Sigma$ by the three forbidden factors $ac,ba,cb$. We claim that
$\Sigma$ is a Möbius number system for the iterative system
$\{F_a,F_b,F_c\}$. 

\picture{The three parabolic maps system}{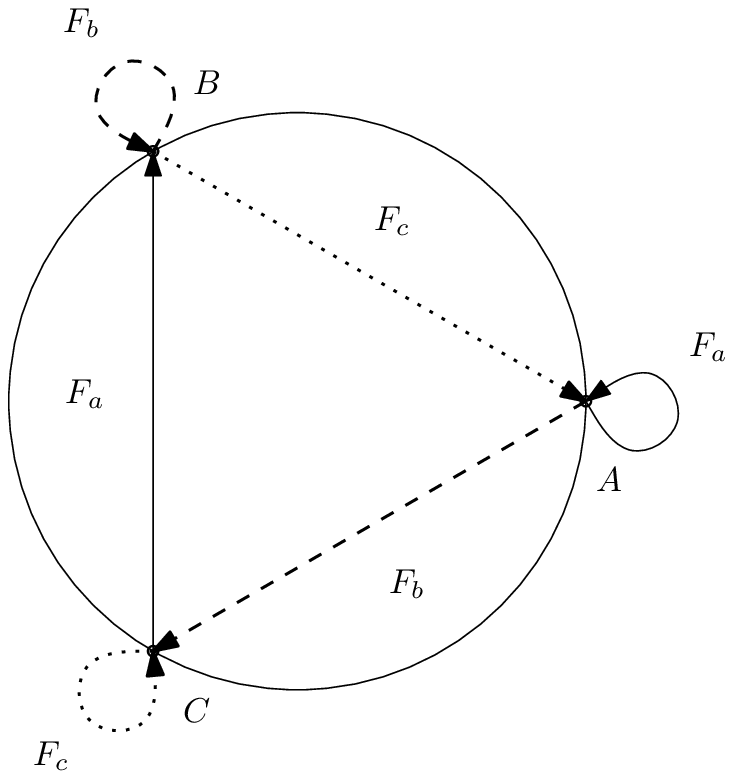}{fig3parabolic}
\end{example}

The following two examples are due to Petr K\r urka, see~\cite{Kurka2}:

\begin{example}\label{expRCF}
The connection between MTs and continued fraction systems is well known. 
We show how to implement continued fractions as a Möbius number system.
Let us take the following three transformations:
\begin{eqnarray*}
\hat{F}_{\one}(x)&=&x-1\\
\hat F_0(x)&=&-\frac{1}x\\
\hat F_1(x)&=&x+1.
\end{eqnarray*} 
These transformations are not disc preserving; instead, they preserve the
upper half plane (representing $\er\cup\{\infty\}$ instead of $\Te$). 
Conjugating $\hat{F}_{\one},\hat F_0,\hat F_1$ with the stereographic projection, we obtain
the following three disc preserving transformations:

\begin{eqnarray*}
F_{\one}&=&
\frac12
\begin{pmatrix}
2-i&-1\\
-1&2+i\\
\end{pmatrix}\\
F_0&=&\begin{pmatrix}
-i&0\\
0&i\\
\end{pmatrix}\\
F_1&=&
\frac12
\begin{pmatrix}
2+i&1\\
1&2-i\\
\end{pmatrix}
\end{eqnarray*}
Words $00,1\one$ and $\one1$  correspond to the identity maps while
$\Phi((01)^\infty)$ and $\Phi((0\one)^\infty)$ are not defined (as $F_{01}$ and $F_{0\one}$
are parabolic). This is why we define the shift $\Sigma$ by the set of
forbidden words $00,1\one,\one1,101,\one0\one$.

It turns out that $\Sigma$ is the regular continued fraction system as depicted
in Figure~\ref{fig-cf}. In this picture, the labelled points represent the
images of the point $0$
under the corresponding sequence of transformations, while curves connect images of $0$
that are next to each other in a given sequence. Observe that the images of $0$
converge to the boundary of the disc, ensuring convergence.

A slight complication not present in the usual continued fraction system is
that we need to juggle with signs, using the transformation $-1/x$ instead of
$1/x$ because the latter does not preserve the unit disc (the map $x\mapsto 1/x$ preserves the
unit circle but turns the disc inside out). Otherwise, the function
$\Phi_{|\Sigma}:\Sigma\to\Te$ 
mirrors the usual continued fraction numeration process.

\zoomedpicture{The regular continued fractions (as appear in Figure 3 in
\cite{Kurka2})}{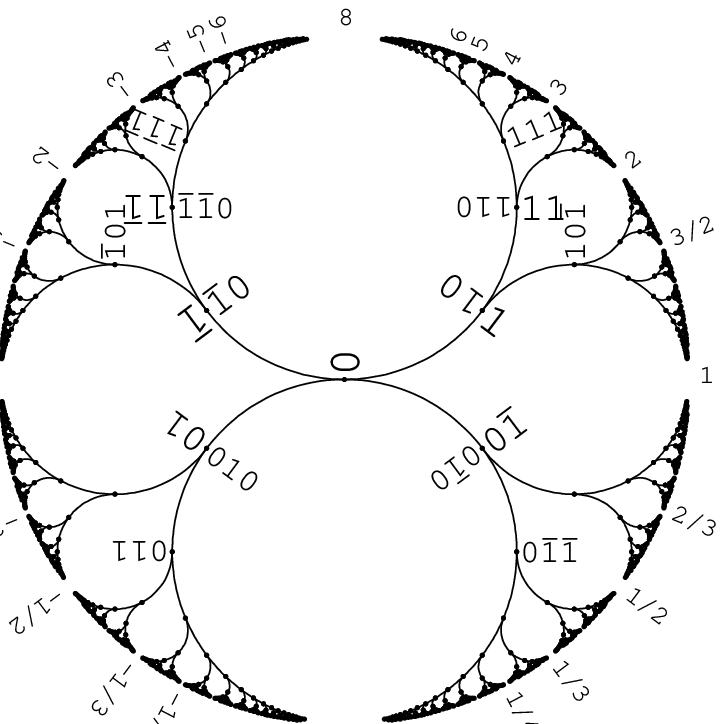}{fig-cf}{4 in}
\end{example}

\begin{remark}
Even a quick glance on Figure~\ref{fig-cf} reveals that parts of the circle
seem to be missing. While $\Phi$ is indeed surjective, the convergence of the
images of $0$ is sometimes quite slow in this system and so the depth used in
the computer graphics was not enough to get near certain points. We can improve the
speed of convergence by adding more transformations like in \cite{Kurka2}.
\end{remark}

\begin{example}\label{expbinary}
As a last example, we obtain a circle variant of the signed binary number
system. Take the following four upper half plane preserving
transformations:

\begin{eqnarray*}
\hat F_{\one}(x)&=&(x-1)/2\\
\hat F_0(x)&=&x/2\\
\hat F_1(x)&=&(x+1)/2\\
\hat F_2(x)&=&2x.
\end{eqnarray*}
Again, we conjugate these MTs with the stereographic projection to be disc preserving:
\begin{eqnarray*}
F_{\one}(x)&=&
\frac1{2\sqrt{2}}
\begin{pmatrix}
3-i&-1-i\\
-1+i&3+i\\
\end{pmatrix}\\
F_0(x)&=&
\frac1{2\sqrt{2}}
\begin{pmatrix}
3&-i\\
i&3\\
\end{pmatrix}\\
F_1(x)&=&
\frac1{2\sqrt{2}}
\begin{pmatrix}
3+i&1-i\\
1+i&3-i\\
\end{pmatrix}\\
F_2(x)&=&
\frac1{2\sqrt{2}}
\begin{pmatrix}
3&i\\
-i&3\\
\end{pmatrix}\\
\end{eqnarray*}

We take these transformations as our iterative system and then define the shift
$\Sigma\subset \{0,1,\one,2\}^\omega$ by forbidding the words
$20,02,12,\one2,1\one$ and $\one1$. 

Why are we forbidding these words?
The reason for disallowing $2$
and $0$ next to each other is that these transformations are inverse to each
other. The first four forbidden pairs ensure that twos are going to appear only at
the beginning of any word, making the system easier to study,
 while the last two forbidden words ensure continuity of
the function $\Phi$ at $2^\infty$ and, as a side-effect, make the
representation nicer (in cryptography, for example, we often wish to
only deal with redundant representations of integers without $1$ and $\one$ next to each other).

The result is the Möbius number system depicted in Figure~\ref{fig:binary}. On
$[-1,1]$, this is essentially the redundant binary system with $\one$ playing the
role of the digit $-1$. To represent numbers outside of $[-1,1]$, we use $F_2$. 

\zoomedpicture{The Möbius signed binary system (as appears in Figure 1 in
\cite{Kurka2})}{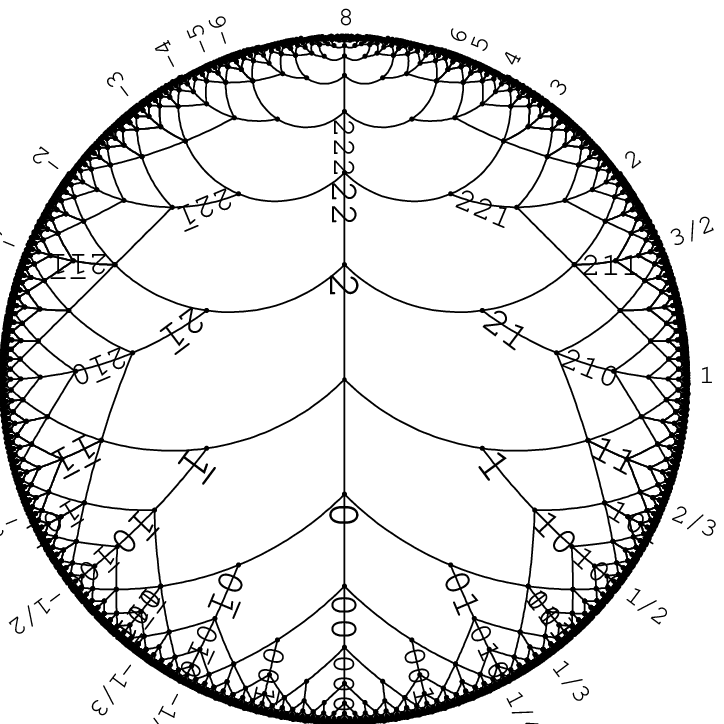}{fig:binary}{4 in}
\end{example}

\subsection{Systems defined by intervals}\label{secIntervalSystems}

This is the main part of our paper. Our goal here is to obtain sufficient
conditions guaranteeing that an iterative Möbius system together with a
subshift form a Möbius number system. All the subshifts we consider here are
defined using the notion of interval almost cover.

\begin{definition}
Let $\{F_a:a\in A\}$ be a Möbius iterative system.
\emph{Interval almost cover} of $\Te$ is any family $\We=\{W_a:a\in A\}$ of sets
such that $\bigcup\{\oW_a:a\in A\}=\Te$ and each $W_a$ is an union of finitely
many open intervals on $\Te$.
\end{definition}

Given an interval almost cover $\We$ and $u\in \Aplus$, we define the \emph{refined set} $W_u$ as
\[
W_u=W_{u_0}\cap F_{u_0}(W_{u_1})\cap\dots\cap F_{u_{[0,|u|-1)}}(W_{|u|-1}).
\]

Note that each refined set is again a finite (possibly empty) union of open intervals.
For formal reasons, let $W_\lambda=\Te$. 
It is easy to prove by induction that for all $u,v\in\Astar$, we have
$W_{uv}=W_u\cap F_u(W_v)$.
 
\begin{lemma}\label{lemExpansion}
Let $W_a\subset V_a$ for every $a\in A$. Then $W_u\subset V_u$ for any $u\in\Astar$.
\end{lemma}
\begin{proof}
We proceed by induction on the length of $u$.

For $|u|=1$, the claim is obvious. Let $u=va$ with $W_v\subset V_v$. Then
$W_{va}=W_v\cap F_v(W_a)$ and $\dot{(F_{va}^{-1})}(x)=\dot{(F^{-1}_{a})}(F^{-1}_v(x))\cdot \dot
{(F^{-1}_v)}(x)$. When $x\in W_{va}$, we have $x\in W_v\subset V_v$ and $F^{-1}_v(x)\in
W_a\subset V_v$, therefore $\dot{(F^{-1}_{a})}(F^{-1}_v(x)),\dot{(F^{-1}_v)}(x) >1$, proving the lemma.
\end{proof}

\begin{definition}
Let $\We$ be an interval almost cover. A subshift $\Sigma$ is \emph{compatible} with
$\We$ if for every $v\in\el(\Sigma)$ it is true that
\[
\oW_v= \bigcup_{a,\,va\in\el(\Sigma)}\oW_{va}.
\]
\end{definition}

Note that the ``$\supset$'' inclusion in the last equality is trivial as
$\oW_{va}\subset \oW_v$ for all $a$.

The meaning of the compatibility condition is, roughly speaking, that we can
safely extend words of $\el(\Sigma)$. Note in particular that $\Aomega$ is compatible
with all the interval almost covers on $A$:
\[
\bigcup_{va\in\el(\Sigma)}\oW_{va}=\bigcup_{va\in\el(\Sigma)}\left(
\oW_v\cap F_v(\oW_a)
\right)=\oW_v\cap F_v \left(\bigcup_{a\in A)}\oW_a\right)=\oW_v\cap
F_v(\Te)=\oW_v
\]

Let $\We$ be an interval almost cover and $\Sigma$ a subshift compatible with $\We$.
We then define several other entities:
\begin{eqnarray*}
\Sigma_\We&=&\{w\in \Sigma:\forall n,\, W_{w_{[0,n)}}\neq \emptyset\}\\
q(u)&=&\min\{\dot{(F^{-1}_{u})}(x): x\in \oW_u\}\\
Q_n(\We,\Sigma)&=&\min\{q(u):|u|=n,u\in\el(\Sigma_\We)\}
\end{eqnarray*}

We will call $\Sigma_\We$ the \emph{interval shift} corresponding to $\Sigma$
and $\We$ (note that $\Sigma_\We$ depends on both). Obviously,
$\Sigma_\We\subset\Sigma$. We show that $\Sigma_\We$ is a subshift:

\begin{enumerate}
\item Let $w\in\Sigma_\We$. Then $\sigma(w)\in\Sigma$ and for all $n$ we have
$\emptyset \neq W_{w_{[0,n)}}\subset F_{w_0}(W_{w_{[1,n)}})$, implying
$W_{w_{[1,n)}}\neq \emptyset$.
Therefore, $\sigma(w)\in\Sigma_\We$.
\item Let $\{w^{(n)}\}_{n=1}^\infty$ be a sequence of words in $\Sigma_\We$ with
the limit $w$. As $\Sigma$ is closed, $w\in\Sigma$. Moreover, for every $k$ there exists an $n$ such that
$w_{[0,k)}=w^{(n)}_{[0,k)}$ and so $W_{w_{[0,k)}}=W_{w^{(n)}_{[0,k)}}\neq \emptyset$. Therefore
$w\in\Sigma_\We$ and so $\Sigma_\We$ is closed.
\end{enumerate}

\begin{remark} We now make an observation that will be useful later.
Consider two interval almost covers $\We=\{W_a:a\in A\}$ and $\We'=\{W'_a:a\in
A\}$ and two subshifts $\Sigma'\subset\Sigma$ such that $\We$ is compatible
with $\Sigma$ and $\We'$ is compatible with $\Sigma'$. If for each $a\in A$ we have
$W'_a\subset W_a$ then easily $\Sigma'_{\We'}\subset \Sigma_{\We}$. 
\end{remark}

Our main goal will be to prove the following theorem:
\begin{theorem}\label{thmClosedCover1}
Let $\{F_a:a\in A\}$ be a Möbius iterative system. Assume that $\We$ is such an
interval almost cover that $W_a\subset V_a$ for all $a\in A$ and $\Sigma$ is a subshift
compatible with $\We$. Then
$\Sigma_\We$ is a Möbius number system for the iterative system $\{F_a:a\in A\}$.

Moreover, for every $v\in \Astar$, $\Phi([v]\cap\Sigma_\We)=\oW_v$.
\end{theorem}

Our plan is to first prove several auxiliary claims, then solve the case when
$W_a=V_a$ and finally use this special case to prove
Theorem~\ref{thmClosedCover1}.

The following lemma is stated (in a different form) in \cite{Kurka} as
Lemma 2.

\begin{lemma}\label{lemGrowingInterval}
Let $\{F_a:a\in A\}$ be a Möbius iterative system and let $L$ be the length of the
longest of intervals in $\{\oV_a:a\in A\}$. Then there exists an increasing
continuous function $\psi:[0,L]\to\er$ such that:
\begin{enumerate}
\item $\psi(0)=0$
\item For every $l\in(0,L]$ we have $\psi(l)>l$.
\item If $I$ is an interval and $a\in A$ a letter such that $I\subset \oV_a$ then
$|F^{-1}_a(I)|\geq \psi(|I|)$.
\end{enumerate}
\end{lemma}

\begin{proof}
Thanks to Lemma~\ref{lemUniForm} we can without loss of generality assume that
each $F_a$ is a contraction to 1 with parameter $r_a>1$ (if $r_a=1$ then
$\oV_a=\emptyset$ so $a$ can be safely omitted from the alphabet). 

Choose $a\in A$ so that $r_a$ is minimal and let
\[
\psi(l)=\inf\{|F^{-1}_a(I)|\,:\, I\subset \oV_a\; \mathrm{and}\; |I|=l\}.
\]

By analyzing contractions, it is easy to see that
$\psi$ is increasing, continuous and $\psi(l)>l$ for $l>0$.
\end{proof}

\begin{lemma}\label{lemBrokenInterval}
Let $\{F_a:a\in A\}$ be a Möbius iterative system. Let $x,y$ be points and $w\in\Aomega$ a word such that
for all $n$ we have $F^{-1}_{w_{[0,n)}}(x),F^{-1}_{w_{[0,n)}}(y)\in \oV_{w_n}$. Then $x=y$
or there exists an $n_0$ such that  $F_{w_{[n,n+1]}}$ is a rotation for all
$n>n_0$. 
\end{lemma}
\begin{proof}
Denote $x_n=\Finw(x)$ and $y_n=\Finw(y)$. Assume that $x\neq y$. For each $n$
denote by $I_n$ the closed interval with endpoints $x_n,y_n$ such that
$I_n\subset\oV_{w_n}$. Möbius transformations are bijective, so $x_n\neq y_n$,
implying $|I_n|>0$ for all $n$.

Observe first that the sequence $\{|I_n|\}_{n=0}^\infty$ is a nondecreasing one. For any
particular $n$ we have two possibilities: Either $F^{-1}_{w_n}(I_n)=I_{n+1}$ and therefore
$|I_{n+1}|\geq \psi(|I_n|)>|I_n|$ (by Lemma~\ref{lemGrowingInterval}), or
$I_{n+1}=\overline{\Te \setminus F^{-1}_{w_n}(I_n)}$;
see Figure~\ref{figOverflow}. In the second case, observe that
$\Te\setminus U_{w_n}\subset I_{n+1}$ and recall that
$|U_{w_n}|+|V_{w_n}|=2\pi$. It follows  $|I_{n+1}|\geq |V_{w_n}|\geq|I_n|$.
Call the second case a \emph{twist}.

\zoomedpicture{A twist.}{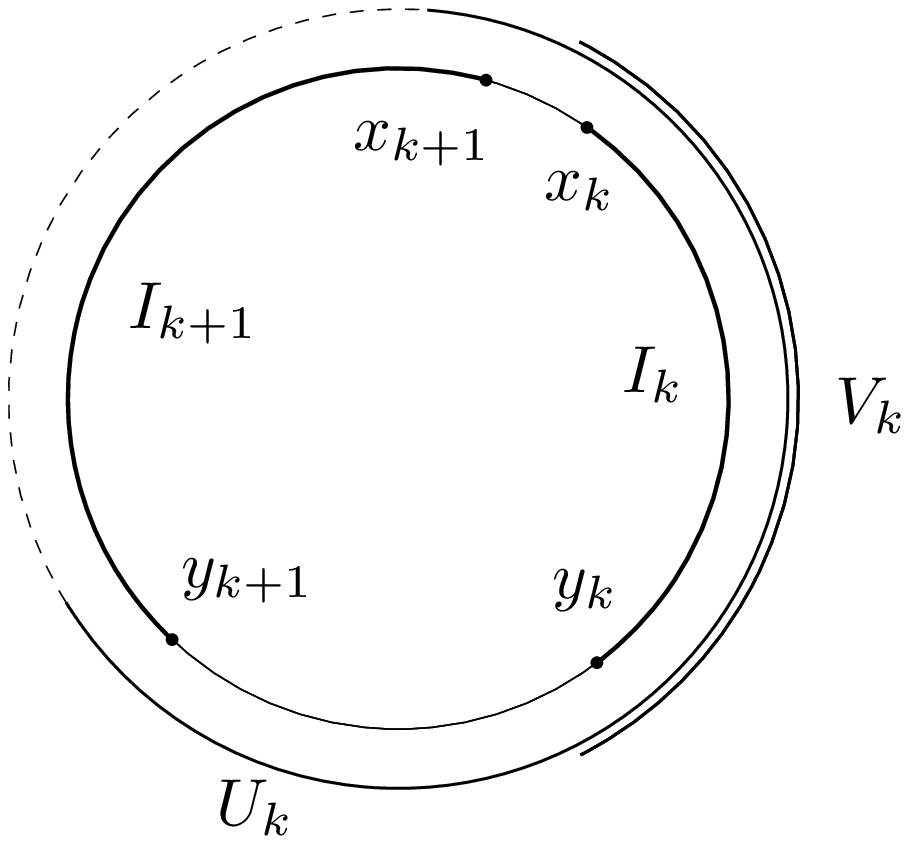}{figOverflow}{3 in}

Assume first that the number of twists is infinite. We claim that
then there must
exist an $n_0$ such that $|I_{n}|$ is constant for all $n>n_0$: Whenever 
a twist happens for some $n$, we have $|V_{w_{n+1}}|\geq|I_{n+1}|\geq
|V_{w_n}|$. If any inequality in the previous formula is sharp then the letter $w_n$ does not appear
anywhere in $w_{[n+1,\infty)}$ anymore, so if there were infinitely many such sharp
inequalities, the alphabet $A$ would have to be infinite.
Therefore, $|V_{w_{n+1}}|=|I_{n+1}|$ for all but finitely many
twists and the finiteness of $A$ gives us that that there is an $n_0$ such that
$|I_{n}|=|I_{n_0}|$ whenever $n>n_0$.

Assume now $n>n_0$. Simple case analysis shows that $|I_n|=|I_{n+1}|$ can only
happen when $x_n,y_n$ are
the endpoints of $\oV_{w_{n}}$ and the transition is a twist. Similarly,
$x_{n+2}, y_{n+2}$ are also endpoints of $\oV_{w_{n+1}}$. Let $R$ be the
rotation that sends $I_{n+2}$ to $I_n$. The map $R\circ F^{-1}_{w_{[n,n+1]}}$ has two fixed points $x_n,y_n$ and
\[
\dot{(R\circ F^{-1}_{w_{[n,n+1]}})}(x_n)=\dot{(R\circ
F^{-1}_{w_{[n,n+1]}})}(y_n)=1.
\]
 This can happen only when
$R\circ F^{-1}_{w_{[n,n+1]}}=\id$ (see the Remark after Definition~\ref{defMT}). Therefore, $F^{-1}_{w_{[n,n+1]}}$ is a
rotation for all $n>n_0$ and we are done.

It remains to investigate the case when the number of twists is finite, i.e.
there exists an $n_0$ such that for all $n\geq n_0$ we have
$I_{n+1}=F^{-1}_{w_n}(I_n)$. 
By Lemma~\ref{lemGrowingInterval} we obtain $|I_{n+1}|\geq \psi(|I_n|)$ for all
$n\geq n_0$ and therefore $|I_n|\geq \psi^{n-n_0}(|I_{n_0}|)$ for all $n\geq
n_0$. Denote $l=|I_{n_0}|>0$.

Consider now the sequence $\{\psi^{n}(l)\}_{n=0}^{\infty}$. Assume that 
$\psi^n(l)\leq L$ for all $n$. Then the sequence is increasing and bounded and
therefore it has a limit $\xi\in(0,L]$. As $\psi$ is continuous, $\psi(\xi)=\xi$. 
But the only fixed point of $\psi$ is $0$, a contradiction.

Therefore, there always exists an $n$ such that $\psi^{n}(l)>L$. But then
$I_{n+n_0}$ cannot possibly fit into any of the intervals $\oV_a$, which is a
contradiction with the assumption $x_{n+n_0},y_{n+n_0}\in \oV_{w_{n+n_0}}$. Therefore, $x=y$.
\end{proof}

Let $I$ be an interval on $\Te$. Recall that if $I=(a,b)$ then $a$ is the
clockwise and $b$ the counterclockwise endpoint of $I$. The following two easy
observations on the geometry of intervals are going to be useful when examining
our number system.

\begin{lemma}\label{lemIntersection}
Let $I_1,\dots, I_k$ be open intervals on $\Te$. Then $x\in\overline{\bigcap_{i=1}^k
I_i}$
if and only if both of the following conditions hold: 
\begin{enumerate}
\item $x\in \bigcap_{i=1}^k\overline {I_i}$
\item If $x$ is an endpoint of both $I_i$ and $I_j$ then $x$ may not be the counterclockwise endpoint 
of one interval and clockwise endpoint of the other.
\end{enumerate}
\end{lemma}

\begin{proof}
Obviously, if $x\in \overline{\bigcap_{i=1}^k I_i}$ then
$x\in\bigcap_{i=1}^k\overline {I_i}$.
Were $x$ the clockwise endpoint of $I_i$ and counterclockwise endpoint of $I_j$
then there would exist a neighborhood $E$ of $x$ such that
$E\cap I_i\cap I_j=\emptyset$ and so $x$ would not belong to $\overline{\bigcap_{i=1}^k I_i}$.

In the other direction, assume that $x$ has both the required properties. 
Then a simple
case analysis shows that for any neighborhood $E$ of $x$ we have $E\cap 
\bigcap_{i=1}^k I_i \neq \emptyset$ and so $x\in\overline{\bigcap_{i=1}^k I_i}$.
\end{proof}

\begin{lemma}\label{lemCircleCover}
Let $J=[x,y]$ be a nondegenerate interval on $\Te$.
Let $I_1,\dots,I_k$ be closed intervals such that $J\subset\bigcup_{i=1}^k
I_i$. Then there exists $i$ such that $x$ is not the counterclockwise endpoint
of $I_i$, that is $[x,x+\epsilon]\subset I_i$ for some $\epsilon>0$.
\end{lemma}
\begin{proof}
Let $I_i=[a_i,b_i]$ for all $i$. 

If for some $i$, $x\in [a_i,b_i)$, we are done. We
have $x\not\in I_i$ or $x=b_i$ for each $i$. But then there exists $\epsilon$ such that
$|J|>\epsilon>0$ and
$(x,x+\epsilon)\cap I_i=\emptyset$ for all $i$, a contradiction.
\end{proof}

We are going to construct a candidate for the
graph $X=\{(\Phi(w),w):w\in\Omega\}$ of $\Phi_{|\Omega}$.
During the proof, the set $\Omega$ turns out to be a subshift.

Given $(x,w)\in\Te\times\Aomega$, we use the shorthand notation
$x_i=F^{-1}_{w_{[0,i)}}(x)$. For $a\in A$, label
$e^+_a$ the counterclockwise and $e^-_a$ the clockwise endpoint of the interval
$V_a$. 

Define $X\subset \Te\times\Aomega$ 
to be the set of all pairs $(x,w)$ such that: 
\begin {enumerate}
\item[(1)] For all $i=1,2,\dots,\, x_i \in \oV_{w_{i}}$.
\item[(2)] For no $i$ and $j$ is it true that $x_i=e^+_{w_i}$ and
$x_j=e^-_{w_{j}}$.
\end{enumerate}
\zoomedpicture{A situation contradicting part (2) of the definition of
$X$}{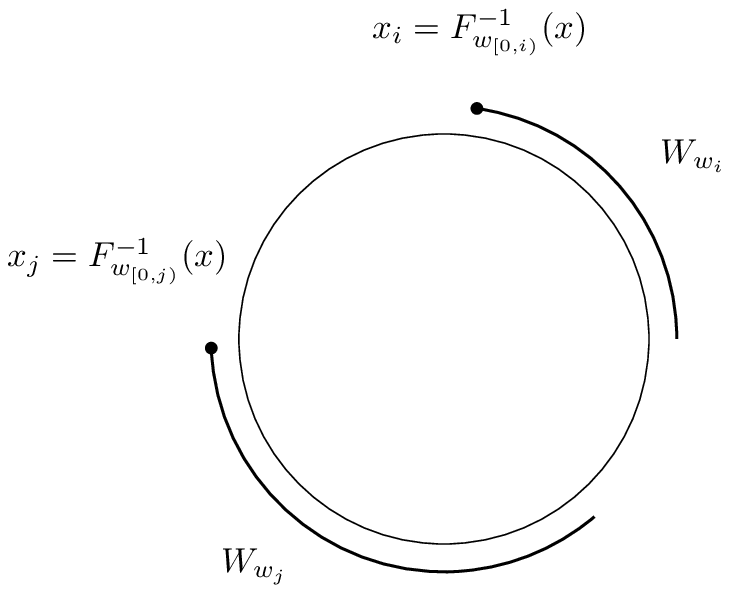}{figDesired}{4in}
Note that the second condition says that endpoints cannot ``alternate'':
If $x_i,x_j$ are endpoints of $V_{w_i},V_{w_j}$ then 
$x_i,x_j$ are both endpoints of the same type (clockwise or
counterclockwise). For an example of a forbidden situation, see Figure~\ref{figDesired}. 

Preparing for the future, we investigate the set $X$. 
Using Lemma~\ref{lemIntersection}, it is easy to see that $(x,w)\in X$ if and
only if $x\in\bigcap_{k=0}^\infty
\overline{\bigcap_{i=0}^k
F_{w_{[0,i)}}(V_{w_{i}})}.$ Remembering the definition of refined sets, we
let $W_{u}=\bigcap_{i=0}^{k-1} F_{u_{[0,i)}}(V_{u_{i}})$  for $u\in A^k$. 

We now have:
\[
\bigcap_{k=0}^\infty
\overline{\bigcap_{i=0}^k
F_{w_{[0,i)}}(V_{w_{i}})}=
\bigcap_{k=1}^\infty \overline{W_{w_{[0,k)}}}
\]

Note that there is in general a difference between $W_{w_{[0,k)}}$ and 
$V_{w_{[0,k)}}$. The former is defined as an intersection of preimages of
intervals $V_{w_i}$,
while the latter is the expanding interval of $F^{-1}_{w_{[0,k)}}$. These two sets
are different in general, but Lemma~\ref{lemExpansion} gives us the inclusion $W_{w_{[0,k)}}\subset V_{w_{[0,k)}}$.

Let $P(v)=\overline{W_{v}}$. We have $(x,w)\in X$ iff
$x\in\bigcap_{k=1}^\infty P(w_{[0,k)})$.

\begin{observation}
\label{obsNotRotation}
If $(x,w)\in X$, where $X$ is defined as above,
then $F_{w_{[k,k+1]}}$ is not a rotation for any $k$.
\end{observation}
\begin{proof}
From the condition $(x,w)\in X$ iff $x\in\bigcap_{k=1}^\infty P(w_{[0,k)})$ we
obtain that the set $P(w_{[0,k+1]})=\overline{W_{w_{[0,k+1]}}}$ is nonempty. Therefore,
$W_{w_{[k,k+1]}}\neq \emptyset$. By Lemma~\ref{lemExpansion},
$V_{w_{[k,k+1]}}\neq\emptyset$ and therefore $F_{w_{[k,k+1]}}$
cannot be a rotation.
\end{proof}

We are now ready to prove a special case of Theorem \ref{thmClosedCover1}.

\begin{theorem}\label{thmClosedCover2}
Let $\{F_a:a\in A\}$ be a Möbius iterative system. For $a\in A$ let $V_a$ be the expansive
interval of $F_a^{-1}$. Assume that $\Ve=\{V_a:a\in
A\}$ is an interval almost cover of $\Te$ and let $\Omega=\left(\Aomega\right)_{\Ve}$
be the corresponding interval shift. 
Then $\Omega$ is a Möbius number system for the iterative system $\{F_a:a\in A\}$.
\end{theorem}

\begin{proof}
If $V_a=\emptyset$ for some $a\in A$ then the letter $a$ does not
appear in the shift $\Omega$ at all. Therefore, without loss of generality
$V_a\neq\emptyset$ for all $a\in A$.

Consider the set $X\subset \Te\times\Aomega$ as introduced above. Taking projections 
$\pi_1,\pi_2$ of $X$ to the first and second element, we obtain the set of 
points $\pi_1(X)\subset \Te$ and the set of words $\pi_2(X)\subset\Aomega$. 
To conclude our proof, we verify that:
\begin{enumerate}
\item $\pi_1(X)=\Te$
\item $\pi_2(X)=\Omega$
\item For $(x,w)\in X$ we have $\Phi(w)=x$.
\item $\Phi_{|\Omega}$ is continuous.
\end{enumerate}
\begin{enumerate}
\item This follows from the fact that $\{\oV_a:a\in A\}$ covers
$\Te$. Given $x\in\Te$ we can construct $w\in\Aomega$ satisfying $(x,w)\in
X$ by induction using Lemma~\ref{lemCircleCover}: Let $x_0=x$ and choose $w_0\in A$ such that $x_0\in \oV_{w_0}$ and $x_0$ is not
the counterclockwise endpoint of $\oV_{w_0}$. Then take $x_1=F^{-1}_{w_0}(x_0)$,
choose $w_1$ so that $x_1\in\oV_{w_1}$ and $x_1$ is not the counterclockwise
endpoint of $\oV_{w_1}$, let $x_2=F^{-1}_{w_2}(x_1)$ and repeat the procedure.
\item 
Notice that $W_v\neq \emptyset$ iff $P(v)\neq\emptyset$.  and
$P(w_{[0,k+1)})\subset P(w_{[0,k)})$. By compactness of $\Te$, 
$w\in\Omega$ if and only if there exists $x\in\bigcap_{k=1}^\infty
P(w_{[0,k)})$. But $x\in\bigcap_{k=1}^\infty P(w_{[0,k)})$ iff $(x,w)\in X$.
Therefore $w\in\Omega$ iff $w\in\pi_2(X)$.

\item Let $l=\min\{|\oV_a|: a\in A\}$ and assume $(x,w)\in X$. Recall the notation $x_i=F^{-1}_{w_{[0,i)}}(x)$. 
We will divide the proof into several cases.

If $\lim_{i\to\infty}\dot{(F^{-1}_{w_{[0,i)}})}(x)=\infty$ then we simply use 
Lemma~\ref{corConvergencyTools} and are done.

Therefore, assume
that $\lim_{i\to\infty}\dot{(F^{-1}_{w_{[0,i)}})}(x)$ is finite or does not
exist and examine the consequences.

We have
\[
\dot{(F^{-1}_{w_{[0,i)}})}(x)=\prod_{k=0}^{i-1}
\dot{(F^{-1}_{w_k})}(F^{-1}_{w_{[0,k)}}(x)) =\prod_{k=0}^{i-1}
\dot{(F^{-1}_{w_k})}(x_k).
\]
As $x_k\in \oV_{w_k}$, we have that $\dot{(F^{-1}_{w_k})}(x_k)\geq 1$ for each
$k$. From this we deduce that $\lim_{i\to\infty}\dot{(F^{-1}_{w_{[0,i)}})}(x)=
\prod_{k=0}^{\infty} \dot{(F^{-1}_{w_k})}(x_k)$ exists.

Recall the notation $\oV_a=[e^-,e^+]$. Examining the function
$\dot{(F^{-1}_a)}$ as in the proof of Lemma~\ref{lemMobiology}, we see
that for every $\xi$ such that $2l>\xi>0$ there exists a
$\delta>0$ such that for each $a\in A$ and each $y\in[e^-_a+\xi,e^+_a-\xi]$
we have $\dot{(F^{-1}_{a})}(y)>1+\delta$ (see Figure~\ref{figEndpoints}).
\zoomedpicture{The interval $[e^-_a+\xi,e^+_a-\xi]$ and
friends.}{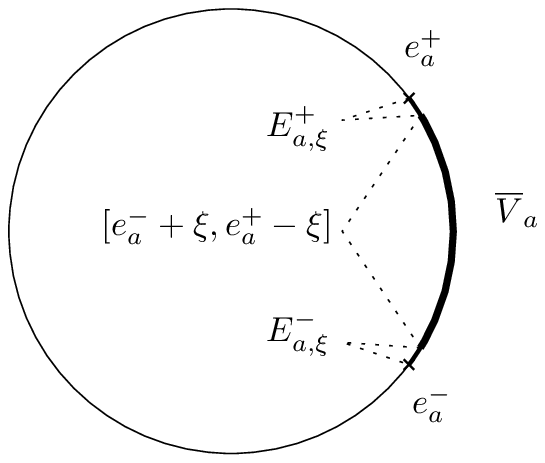}{figEndpoints}{3 in}
 
As $\lim_{k\to \infty}(1+\delta)^k=\infty$ for any $\delta>0$, the only way
that the product $\prod_{k=0}^{\infty} \dot{(F^{-1}_{w_k})}(x_k)$
can be finite is when for each $\xi$ such that $2l>\xi>0$ there exists $k_0$ such 
that
\[
k\geq k_0\Rightarrow
x_k\in[e^-_{w_k},e^-_{w_k}+\xi)\cup(e^+_{w_k}-\xi,e^+_{w_k}].
\]

Let 
$E^-_{a,\xi}=[e^-_a,e^-_a+\xi)$ and $E^+_{a,\xi}=(e^+_a-\xi,e^+_a]$. As the set 
$\{e^-_a,e^+_a:a\in A\}$ is finite, if we choose $\xi$ small enough then for any $a,b\in A$ and any choice of
$\triangle,\square$ in $\{+,-\}$, we obtain:
\[
F_{a}^{-1}(E^\triangle_{a,\xi})\cap E^{\square}_{b,\xi}\neq\emptyset\,\Rightarrow\,
F^{-1}_{a}(e^\triangle_{a})=e^{\square}_{b}.
\]
Moreover, if $\triangle$ and $\square$ are different then
\[
F_{a}^{-1}(E^\triangle_{a,\xi})\cap E^{\square}_{b,\xi}=\{e^\square_b\}.
\]

Take such a small $\xi$. If 
$\Phi(w_{[k_0,\infty)})=x_{k_0}$ then
$\Phi(w)=F_{w_{[0,k_0)}}(\Phi(w_{[k_0,\infty)})=x$. Therefore, we can without
loss of generality assume that $k_0=0$, i.e.:
\[
\forall k\geq 0,\,
x_k\in E^-_{{w_k},\xi}\cup E^+_{{w_k},\xi}
\]

Assume that $x_{k}\in E^+_{w_k,\xi}$ and $x_{k+1}\in E^-_{w_{k+1},\xi}$. By the choice
of $\xi$, this would only be possible when $x_k=e^+_{w_k}$ and
$x_{k+1}=e^-_{w_{k+1}}$, a contradiction with $(x,w)\in X$.

Therefore, without loss of generality, $x_k\in E^-_{w_k,\xi}$ for each $k\geq
0$. Let $e_k=F^{-1}_{[0,k)}(e^-_{w_0})$ and observe that in this case
$e_k=e^-_{w_k}$ for each $k\geq0$.

Now $x_k$ and $e_k$ both belong to $\oV_{w_k}$ for all
$k\geq 0$ and, using
Lemma~\ref{lemBrokenInterval} and Observation~\ref{obsNotRotation}, we obtain
$x_k=e_k=e^-_{w_k}$ for all $k\geq 0$. 

We will conclude the proof using part (3) of  Theorem~\ref{thmCharacteriseConvergence}. Let $J$ be an open interval containing
$x$. Then there exists $\epsilon>0$ such that $I=[x,x+\epsilon]\subset J$.
We want to show that 
\[
\liminf_{k\to\infty}|F^{-1}_{w_{[0,k)}}(I)|\geq l,
\]
where
again $l=\min\{|\oV_a|: a\in A\}$. Observe that each  $F^{-1}_{w_{k}}$ expands the
interval $[x_k,x_k+l]$, therefore it is enough to find a
single $k$ such that $|F^{-1}_{w_{[0,k)}}(I)|\geq l$.

But this again follows from Lemma~\ref{lemBrokenInterval} and
Observation~\ref{obsNotRotation}; all we have to do is 
choose the point $y$ equal to $x+\epsilon$. Obviously, $y\neq x$, so there must
exist $k$ such that $F^{-1}_{w_{[0,k)}}(y)\not \in \oV_{w_k}$ which is only
possible when $|F^{-1}_{w_{[0,k)}}(I)|>l$.

Therefore, by Theorem~\ref{thmCharacteriseConvergence}, we have
$\Phi(w)=x$, concluding the proof.

\item 
We begin by proving that $\bigcap_{i=1}^\infty P(w_{[0,i)})$ is a
singleton for every $w\in \Omega$. We know that when $(x,w)\in X$ then $x\in
P(w)$. Let $y\in P(w)$. Then for all
$k$ we have $F^{-1}_{w_{[0,k)}}(\{x,y\})\in \oV_{w_k}$ and, applying Lemma
\ref{lemBrokenInterval} together with Observation~\ref{obsNotRotation}, we obtain $x=y$,
which is what we need.

Now given an open set $U\subset \Te$ we want to show that $\Phi^{-1}(U)$ is open in
$\Omega$. Let $w\in \Phi^{-1}(U)$ and assume that for all $k$ there exists
$w'\in\Omega$
such that $w_{[0,k]}=w'_{[0,k]}$ and $\Phi(w')\not \in U$. Then we must have
$\bigcap_{i=1}^{k}P(w_{[0,i]})\cap U^c\neq \emptyset$ and from compactness of $X$
we obtain that $\cap_{i=1}^{\infty}P(w_{[0,i]})\cap U^c\neq \emptyset$. But
then $\cap_{i=1}^{\infty}P(w_{[0,i]})$ cannot be a singleton as $\Phi(w)\not\in
U^c$, a contradiction.\qedhere
\end{enumerate}
\end{proof}
Note that the theorem is also true for a slightly larger subshift than
$\Omega$. For details, see \cite{Kazda}, Theorem 21.

Finally, we have prepared the groundwork for the proof of Theorem~\ref{thmClosedCover1}.
\begin{thmCl}
Let $\{F_a:a\in A\}$ be a Möbius iterative system. Assume that $\We$ is such an
interval almost cover that $W_a\subset V_a$ for all $a\in A$ and $\Sigma$ is a subshift
compatible with $\We$. Then
$\Sigma_\We$ is a Möbius number system for the iterative system $\{F_a:a\in A\}$.

Moreover, for every $v\in \Astar$, $\Phi([v]\cap\Sigma_\We)=\oW_v$.
\end{thmCl}

\begin{proof}
As $W_a\subset V_a$ for each $a\in A$ and $\Sigma\subset\Aomega$, 
we have $\Sigma_\We\subset \Omega$ where $\Omega$ is the subshift from
Theorem~\ref{thmClosedCover2}. Therefore, using
Theorem~\ref{thmClosedCover2}, we obtain that $\Phi$ is defined and continuous
on $\Sigma_\We$.

It remains to prove that $\oW_v=\Phi([v]\cap\Sigma_\We)$ for each $v\in\Astar$
(this will also show
the surjectivity of $\Phi_{|\Sigma_\We}$, as $\oW_\lambda=\Te$). As usual, we prove two
inclusions.

First, $\oW_v\subset\Phi([v]\cap\Sigma_\We)$.
Let $x\in \oW_v$ be any point. We want to find an infinite word $w$ such that
$vw\in\Sigma$ and we have
$x \in \oW_{v\cdot w_{[0,k)}}$ for all $k$. This will be enough, as for such a
$w$ we will have
$vw\in\Sigma_{\We}$ as well as $(x,vw)\in X$. Then by the proof of
Theorem~\ref{thmClosedCover2}, $x=\Phi(vw)$ and so $x\in\Phi([v]\cap\Sigma_\We)$.

Without loss of generality assume that there exists $\epsilon_0>0$ such that
$[x,x+\epsilon_0]\subset \oW_v$ (if this is not true then there exists
$\epsilon_0>0$ such that $[x-\epsilon_0,x]\subset\oW_v$ and the proof is
similar). 

We now inductively construct $w$: Assume that $[x,x+\epsilon_k]\subset \oW_{v\cdot w_{[0,k)}}$ for $\epsilon_k>0$ (this
is true for $k=0$). Then from the compatibility condition, we obtain that 
\[
\oW_{v\cdot w_{[0,k)}}= \ds\bigcup_{v\cdot w_{[0,k)}a\in\el(\Sigma)}\oW_{v\cdot
w_{[0,k)}a}.
\]

Applying Lemma~\ref{lemCircleCover} with $J=[x,x+\epsilon_k]$, we
find $w_k=a$ and $\epsilon_{k+1}>0$ such that
$v\cdot w_{[0,k)}w_k\in\el(\Sigma)$ and $[x,x+\epsilon_{k+1}]\subset W_{v\cdot
w_{[0,k)}w_k}$,
completing the induction step.

It remains to show that $\oW_v\supset \Phi([v]\cap \Sigma_\We)$.
Let $\Phi(vw)=x$ for $vw\in\Sigma_\We$. We claim that then $x\in\oW_v$.

By Theorem~\ref{thmCharacteriseConvergence}, we have that
$\oV_{v\cdot w_{[0,k)}}\to\{x\}$ as $k$ tends to $\infty$. 
From the inclusion $\oW_{v\cdot w_{[0,k)}}\subset \oV_{v\cdot w_{[0,k)}}$, we obtain
$\oW_{v\cdot w_{[0,k)}}\to \{x\}$. To complete the proof, notice that for all $k$ we have
$\oW_{v\cdot w_{[0,k)}}\subset\oW_v$, which is only possible when $x\in \oW_v$.
\end{proof}

\begin{corollary}\label{corClosedCover2}
Let $\{F_a:a\in A\}$ be a Möbius iterative system and $B\subset \Bplus$ a finite
set of words. Assume that $\We$ is such
an interval almost cover and $\Sigma$ such a subshift compatible with $\We$ that
$W_b\subset V_b$ for every $b\in B$ and each $w\in\Sigma_\We$ contains as a
prefix some $b\in B$.

Then $\Sigma_\We$ is a Möbius number system for the iterative system $\{F_a:a\in A\}$.
Moreover, $\Phi([v]\cap \Sigma_\We)=\oW_v$ for all $v\in\Astar$.
\end{corollary}

\begin{proof}
We will take the set $B$ as our new alphabet, solve the problem in $\Bomega$
and then return back to $\Aomega$. Denote $\psi:\Bstar\to\Astar$ the 
map that ``breaks down'' each $b\in B$ into its letters. We can extend $\psi$
in an obvious way to obtain a map $\Bomega\to\Aomega$. Denote this map also by
$\psi$.

Because every $w\in\Aomega$ has a prefix in $B$, the map
$\psi:\Bomega\to\Aomega$ is surjective. A little thought gives us that
$\psi$ is continuous in the product topology on $\Bomega$ and $\Aomega$. 

Let $\Theta=\psi^{-1}(\Sigma)$. We claim that $\Theta$ is a subshift of $\Bomega$.
 As $\psi$ is continuous, $\Theta$ must be closed and
from the shift-invariance of $\Sigma$ easily follows shift-invariance of
$\Theta$. 

Consider the Möbius iterative system $\{F_b:b\in B\}$ and let
$\We^{B}=\{W_b:b\in B\}$. We claim that $\We^B$ is an interval 
almost cover  compatible with the subshift
$\Theta$.

Before we proceed, notice that we have now two meanings for the refined set $W_v$
depending on whether $v\in\Astar$ or $v\in\Bstar$. However, a quick proof by
induction yields that for any $v\in\Bstar$
we have $W_v=W_{\psi(v)}$. This is why we will identify $v$ and
$\psi(v)$ when talking about refined sets.

We show that for every $v\in \Bstar$ we have
\[
\oW_v=\bigcup_{b\in B,\,vb\in\el(\Theta)}\oW_{vb}.
\]
Not only does this prove compatibility, it also shows that $\We^B$ is an
interval almost cover (because $\oW_\lambda=\Te$).

Let $v\in\Bstar$. Let $n=\max\{|b|:b\in B\}$ and choose any $x\in\oW_v$. From the compatibility
of $\We$ and $\Sigma$ we obtain by induction that there exists
a word $w$ of length $n$ such that $x\in \oW_{vw}$ and $vw\in\el(\Sigma)$. But then there
exists a $b\in B$ that is a prefix of $w$. Therefore $\oW_{vb}\subset\oW_{vw}$
and so $x\in\oW_{vb}$ and $vb\in\el(\Theta)$.  We have shown that 
\[
\oW_v=\bigcup_{b\in B,vb\in\el(\Theta)}\oW_{vb}.
\]

As $W_b\subset V_b$ for each $b\in B$, Theorem~\ref{thmClosedCover1} now yields
that $\Theta_{\We^B}$ is a Möbius number system for $\{F_b:b\in B\}$ and that
for every $v\in \Bstar$, $\Phi([v]\cap\Theta_{\We^B})=\oW_v$.

As $\psi(\Theta)=\Sigma$ and for each $v\in\Bplus$, $W_v=\emptyset$ iff 
$W_{\psi(v)}=\emptyset$, we obtain that $\psi(\Theta_{\We^{B}})=\Sigma_\We$.

To prove that $\Sigma_\We$ is a Möbius number system, we still need to 
verify that if $w=\psi(u)$ for $u\in \Theta_{\We^B}$ then the sequence
$\{F_{w_{[0,n)}}\}_{n=1}^\infty$ actually represents $\Phi(u)$. 
Fortunately, this is not difficult: Let 
\[
E=\{F_v(0):\;\hbox{$v$ is a prefix of some $b\in B$}\}.
\]
The set $E$ is  finite (and therefore compact) and lies inside the unit circle. By
Theorem~\ref{thmCharacteriseConvergence}, $F_{u_{[0,k)}}(E)\to \{x\}$ for
$k\to\infty$. Observe that for every $n$ there exist $k$ and $v$ such that
$w_{[0,n)}=\psi(u_{[0,k)})v$ and $v$ is a prefix of some $b\in B$. Therefore, 
$F_{w_{[0,n)}}(0)\in F_{u_{[0,k)}}(E)$ and so $F_{w_{[0,n)}}(0)\to x$. It
follows that the sequence $\{F_{w_{[0,n)}}\}_{n=1}^\infty$ represents $\Phi(u)$.

It remains to show that $\Phi([v]\cap\Sigma_\We)=\oW_v$ for every $v\in\Astar$.
Observe that we already have this result for $v=\psi(u)$ where $u\in\Bstar$.
Again, we prove two inclusions.

To prove $\oW_v\subset \Phi([v]\cap \Sigma_\We)$, consider any $x\in \oW_v$.
Then we can find $z\in\el(\Theta)$ such that $v$ is a prefix of $\phi(z)$ and
$x\in\oW_z$. Therefore:
\[
x\in\Phi([z]\cap \Theta_{\We^B})\subset\Phi([v]\cap \Sigma_\We).
\]

Let $w=\psi(u)$ with $u\in\Theta_{\We^B}$. Observe that $\Phi(w)=x$ iff
$\bigcap_{k=1}^\infty \oW_{\psi(u_{[0,k)})}=\{x\}$. Because the sets
$\oW_{w_{[0,k)}}$ form a chain, we can actually rewrite this condition as
$\bigcap_{n=1}^\infty \oW_{w_{[0,n)}}=\{x\}$. Therefore, $\Phi(w)\in
\oW_{w_{[0,n)}}$ for all $n$. Letting $w_{[0,n)}=v$, we have $\Phi([v]\cap
\Sigma_\We)\subset\oW_v$.
\end{proof}

As the set $B=\{v:v\in\el(\Sigma_\We),|v|=n\}$ contains a prefix
of every $w\in\Sigma_\We$, we obtain an improvement of parts (1)--(3) of 
Theorem 10 in \cite{Kurka3}.

\begin{corollary}\label{corClosedCover1}
Let $\{F_a:a\in A\}$ be a Möbius iterative system. Assume that $\We$ is such
an interval almost cover and $\Sigma$ such a subshift compatible with $\We$
that either
$Q_n(\We,\Sigma)>1$ or $Q_n(\We,\Sigma)= 1$ and 
no $F_v,\,v\in \el(\Sigma_\We)\cap A^n$ is a rotation. Then
$\Sigma_\We$ is a Möbius number system for the iterative system $\{F_a:a\in A\}$.

Moreover, $\Phi([v]\cap \Sigma_\We)=\oW_v$ for all $v\in\Astar$.
\end{corollary}
\begin{proof}
Choose $B=\{v:v\in\el(\Sigma_\We),|v|=n\}$. We need to verify that
$W_b\subset V_b$ for each $b\in B$. If $Q_n(\We,\Sigma)>1$ then even $\oW_b\subset V_b$,
while if $Q_n(\We,\Sigma)=1$ and no $F_b$ is a rotation
then for all $x\in W_b$ we have
$\dot{(F^{-1}_n)}(x)>1$ (the inequality is sharp), therefore $W_b\subset V_b$
and we can apply Corollary~\ref{corClosedCover2}.  
\end{proof}

\subsection{Examples revisited}\label{secExamples2} 

We now return to the three number systems presented at the end of
Section~\ref{secBasic} and prove that they indeed are Möbius number systems.
The main practical advantage of using Theorem~\ref{thmClosedCover1}
and Corollary~\ref{corClosedCover2} is that they 
turn verifying convergence, continuity and surjectivity of
$\Phi_{|\Sigma}$ into a set of combinatorial problems (finding $\We$ and
$\Sigma$, describing $\Sigma_\We$ and finding $B$ so that $W_b\subset V_b$).

First, let us revisit {\bf Example~\ref{exp3parabolic}}. We have the three parabolic 
transformations $F_a,F_b$ and $F_c$. Observe that
$V_a=(A,B)$, $V_b=(B,C)$ and $V_c=(C,A)$. What is more, the interval shift
$\Omega=(\Aomega)_{\Ve}$ defined using the interval almost cover $\Ve=\{V_a,V_b,V_c\}$ is precisely
the shift $\Sigma$ obtained by forbidding the words $ac,ba,cb$. One way to show
this is to first show that $W_{ac}=W_{ba}=W_{cb}=\emptyset$ and then verify
that whenever $u\in A^n$ does not contain any forbidden factor then
$W_u= F_{u_{[0,n-1)}}(W_{w_{n-1}})$.

We can prove the last equality by induction on $n$: For $n=1$ the claim is
trivial, while for $n=2$ we can examine all the (finitely many) cases. 
Assume that the claim is true for some $n$ and let $u\in A^{n+1}$. Then:
\begin{eqnarray*}
W_u&=&W_{u_{[0,n)}}\cap F_{u_{[0,n)}}(W_{u_n})=
F_{u_{[0,n-1)}}(W_{u_{n-1}})\cap F_{u_{[0,n)}}(W_{u_n})\\
&=&F_{u_{[0,n-1)}}(W_{u_{n-1}}\cap F_{u_{n-1}}(W_{u_n}))=
F_{u_{[0,n-1)}}(W_{u_{[n-1,n]}}).
\end{eqnarray*}
Now $W_{u_{[n-1,n]}}=F_{u_{n-1}}(W_{u_n})$
by the induction hypothesis for $n=2$ and so:
\[
W_u=F_{u_{[0,n-1)}}(F_{u_{n-1}}(W_{u_n}))=F_{u_{[0,n)}}(W_{u_n}).
\]
Having obtained $\Sigma=\Omega$, Theorem~\ref{thmClosedCover2} gives us
 that $\Sigma$ is a Möbius number system.

In the case of the continued fraction system from {\bf
Example~\ref{expRCF}}, 
let 
\[
W_{\one}=(i,-1),W_0=(-1,1)\; \hbox{and}\;W_1=(1,i).
\]
 It is
straightforward to see that then $(\Aomega)_\We$ is precisely the subshift
$\Sigma$ defined
by forbidding $00,1\one,\one1,101,10\one$. It remains to choose the set
$B=\{01,0\one,1,\one\}$ (which contains a prefix of every word
$w\in\Sigma$) and verify that $W_b\subset V_b$ for each $b\in B$.

By Corollary~\ref{corClosedCover2},
we conclude that $\Sigma_\We$ is a Möbius number system.

Finally, we analyze the \emph{signed binary system} from {\bf
Example~\ref{expbinary}}. Proving that this system is indeed a
Möbius number system requires a reasonable amount of computation which we have
decided to skip here and present only the main points of the proof.
 
Define the shift $\Sigma_0$ by
forbidding the words $02,20,12$ and $\one2$. Then let $v:\Er\to\Te$ denote
the inverse of the stereographic projection and consider the interval almost cover
$\We$:
\[
W_{\one}=(-1,q^{-}),W_0=(h^{-},h^{+}),W_1=(q^{+},1),W_2=(h^{-},h^{+}),
\]
where
\begin{eqnarray*}
q^{-}&=&v(-1/4)=\frac{-8-15i}{17}\\
q^{+}&=&v(1/4)=\frac{8-15i}{17}\\
h^{-}&=&v(-1/2)=\frac{-4-3i}5\\
h^{+}&=&v(1/2)=\frac{4-3i}5.
\end{eqnarray*}
Next, we should show that $\We$ is compatible with $\Sigma_0$ and that 
$(\Sigma_0)_\We=\Sigma$ is the subshift defined by the forbidden words
$20,02,12,\one2,1\one,\one1$. This follows from the set of identities:
\begin{eqnarray*}
F_2(\oW_2)\cup F_2(\oW_1)\cup F_2(\oW_{\one})&=& \oW_2\\
F_0(\oW_0)\cup F_0(\oW_1)\cup F_0(\oW_{\one})&=& \oW_0\\
F_1(\oW_0)\cup F_1(\oW_1)&=&\oW_1\\
F_{\one}(\oW_0)\cup F_{\one}(\oW_{\one})&=&\oW_{\one}
\end{eqnarray*}
together with
\[
F_1(W_{\one})\cap W_1=F_{\one}(W_1)\cap W_{\one}=\emptyset.
\]
It remains to take $B=\{0,1,\one,21,2\one,22\}$ and check the requirements of
Corollary~\ref{corClosedCover2}. It is easy to see that each
$w\in\Sigma$ contains as a prefix a member of $B$. Finally, a detailed
calculation (which we omit here) will verify that for every $b\in B$ 
we have $W_b\subset V_b$, therefore the signed binary system is a Möbius number
system by Corollary~\ref{corClosedCover2}.

\subsection{The numbers $Q_n(\We,\Sigma), Q(\We,\Sigma)$ and $Q(\Sigma)$}\label{secQ}

In the previous subsection, we have shown that if $Q_n(\We,\Sigma)> 1$ then the interval shift
$\Sigma_\We$ is a Möbius number system. We offer a partial converse to this
statement.

\begin{definition}
For an interval almost cover $\We$, let
$Q(\We,\Sigma)=\lim_{n\to\infty}\sqrt[n]{Q_n(\We,\Sigma)}$.
\end{definition}

\begin{remark}
Observe that $Q_{n+m}(\We,\Sigma)\geq Q_n(\We,\Sigma)\cdot Q_m(\We,\Sigma)$. 
Therefore 
\[
\log Q_{n+m}(\We,\Sigma)\geq \log Q_n(\We,\Sigma)+\log Q_m(\We,\Sigma)
\]
and so, by Fekete's lemma 
(see Lemma~\ref{lemSupadditive} in the Appendix), we have that $Q(\We,\Sigma)$ always exists
and is equal to $\sup_{n\in\en}\sqrt[n]{Q_n(\We,\Sigma)}$. In particular, we
see the inequality 
$Q(\We,\Sigma)^n\geq Q_n(\We,\Sigma)$ for all $n$.
\end{remark}

Obviously, if $Q(\We,\Sigma)>1$ then there exists $n$ such that $Q_n(\We,\Sigma)>1$ and, by
Corollary~\ref{corClosedCover1}, $\Sigma_\We$ is a Möbius number system. Let us
now examine what happens when $Q(\We,\Sigma)<1$.

\begin{theorem}\label{thm-nonexistent3}
Let $\{F_a:a\in A\}$ be a Möbius iterative system. Then there 
is no interval almost cover $\We$ and no subshift $\Sigma$ compatible with $\We$
simultaneously satisfying:
\begin{enumerate}
\item $Q(\We,\Sigma)<1$
\item $\Sigma_\We$ is a Möbius number system for $\{F_a:a\in A\}$.
\item For every $w\in\Sigma_\We$, $\bigcap_{k=1}^\infty
\oW_{w_{[0,k)}}=\{\Phi(w)\}$.
\end{enumerate}
\end{theorem}
\begin{remark}
The condition (3) is a quite reasonable demand. In particular, 
it is satisfied by all systems such that
$\Phi([u]\cap\Sigma_\We)=\oW_u$ for every $u\in\Aplus$.
\end{remark}
\begin{proof}
Assume such $\We$ and $\Sigma$ exist. Choose any number $R$ satisfying the
inequalities $Q(\We,\Sigma)<R<1$.
We first find $w\in\Sigma_\We$ and $x=\Phi(w)$ such that for all $k$ it is true
that $\dot{(F^{-1}_{w_{[0,k)}})}(x)\leq R^k$, then we prove that such a pair may
not exist in any number system.
The method used for the first step is a modification of the approach from
the proof of Lemma 1.1 in \cite{Gurvits}.

For $v\in\el(\Sigma_\We)$ denote 
\[
\alpha(v)=\frac{\min\{\dot{(F^{-1}_v)}(x):x\in \oW_v \}}{R^{|v|}}.
\]
Denote by $v^{(n)}$ the word of length at most $n$ such that
$v^{(n)}\in\el(\Sigma_\We)$ and $\alpha(v^{(n)})$ is
minimal. Label $x^{(n)}$ the minimum point corresponding to $\alpha(v^{(n)})$. 
Observe that 
\[
\alpha(v^{(n)})=\frac{Q_{\left|v^{(n)}\right|}(\We,\Sigma)}{R^{\left|v^{(n)}\right|}}
\leq \left(\frac{Q(\We,\Sigma)}{R}\right)^{|v^{(n)}|}.
\]
Elementary calculus shows that
for $n$ tending to infinity $|v^{(n)}|\to \infty$ and $\alpha(v^{(n)})\to 0$.

The crucial observation is that when $v^{(n)}=uv$ for $u,v\in\Astar$, we have
the inequality
$\dot{(F^{-1}_{u})}(x^{(n)})\leq R^{|u|}$. Assume the contrary. Then:
\[
\alpha(v^{(n)})=\frac{\dot{(F^{-1}_{uv})}(x^{(n)})}{R^{|uv|}}=
\frac{\dot{(F^{-1}_{v})}(F^{-1}_u(x^{(n)}))}{R^{|v|}}\cdot
\frac{\dot{(F^{-1}_{u})}(x^{(n)})}{R^{|u|}}
\]
and as $\dot{(F^{-1}_{u})}(x^{(n)})> R^{|u|}$ we have
\[
\alpha(v^{(n)})>\frac{\dot{(F^{-1}_{v})}(F^{-1}_u(x^{(n)}))}{R^{|v|}}=\alpha(v),
\]
so we should have chosen $v$ instead of $v^{(n)}$.

As the length of $v^{(n)}$ tends to infinity, we can find $w\in\Aomega$ such that each
$w_{[0,k)}$ is a prefix of infinitely many members of
$\{v^{(n)}\}_{n=1}^\infty$. Take the subsequence $\{v^{(n_k)}\}_{k=1}^\infty$
so that $w_{[0,k)}=v^{(n_k)}_{[0,k)}$ for every $k$. As $\Te$ is compact we can
furthermore choose $n_k$ so that $x^{(n_k)}$ converge to some $x$ for
$k\to\infty$.

We want to show that  $\dot{(F^{-1}_{w_{[0,k)}})}(x)\leq R^k$ and $\Phi(w)=x$.

To prove the first claim, fix $k$. Then $w_{[0,k)}$ is the prefix of
$v^{(n_l)}$ for all $l\geq k$. Therefore,
$\dot{(F^{-1}_{w_{[0,k)}})}(x^{(n_l)})\leq R^k$ whenever $l\geq k$. But
the function $\dot{(F^{-1}_{w_{[0,k)}})}$ is continuous and $x^{(n_l)}$
converge to $x$. This means $\dot{(F^{-1}_{w_{[0,k)}})}(x)\leq R^k$.

To prove $\Phi(w)=x$, fix $k$ again. Then whenever $l\geq k,\,x^{(n_l)}\in \oW_{w_{[0,n_l)}}\subset
\oW_{w_{[0,k)}}$.
Because $\oW_{w_{[0,k)}}$ is closed and $x$ the limit of $x^{(n_l)}$, we have
$x\in\oW_{w_{[0,k)}}$. Because $k$ in the above argument was arbitrary, 
$x\in\oW_{w_{[0,k)}}$ for all $k$. By the assumption (3) on $\Sigma$ and
$\We$ we obtain $\Phi(w)=x$. 

It remains to show that such a pair $w$, $x$ can not exist in any Möbius number
system. We will prove that when $\dot{(F^{-1}_{w_{[0,k)}})}(x)\leq R^k$ for all
$k$ then it
is not true that $\oV_{w_{[0,k)}}\to\{x\}$.

We will need a little observation: For any $\gamma>1$ there exists $\delta>0$
such that whenever $x,y\in\Te$ are such that $\rho(x,y)<\delta$ then
$\forall a\in A,\,\dot{(F^{-1}_a)}(y)\leq \gamma \cdot\dot{(F^{-1}_a)}(x)$.

This observation is actually an easy consequence of $\dot{F}(x)>0$ and
$\dot{F}$ being (uniformly) continuous for any MT $F$: The function $\ln(\dot F(x))$ is
continuous, therefore for any $\epsilon>0$ there exists $\delta$ such that
$\rho(x,y)<\delta$ implies:
\begin{eqnarray*}
\ln(\dot F(y))&\leq& \ln(\dot F(x))+\epsilon\\
\dot F(y)&\leq& e^\epsilon\cdot\dot F(x)
\end{eqnarray*}
Letting $F=F^{-1}_a$ and $\epsilon=\ln \gamma$, we obtain some $\delta_a>0$.
Now let $\delta=\min\{\delta_a:a\in A\}$ to prove our observation.

 Choose $\gamma>1$ so that $\gamma R<1$ and let
$\delta>0$ be such that if $\rho(z,y)<\delta$ then
$\forall a\in A,\,\dot{(F^{-1}_a)}(y)\geq \gamma \cdot\dot{(F^{-1}_a)}(z)$.

It will be enough to prove that whenever $\rho(x,y)<\delta$ and $k\geq 1$, the inequality
\[
\dot{(F^{-1}_{w_{[0,k)}})}(y)\leq(\gamma R)^k
\]
holds. As $(\gamma R)^k<1$, we have $(x-\delta,x+\delta)\cap \oV_{w_{[0,k)}}=\emptyset$,
a contradiction with $\Phi(w)=x$ by part (6) of Theorem~\ref{thmCharacteriseConvergence}.

Denote $x_i=F^{-1}_{w_{[0,i)}}(x)$ and $y_i=F^{-1}_{w_{[0,i)}}(y)$ and let
$k\geq 1$.

For $k=1$ we have from the definition of $\delta$: 
\[
\dot{ (F^{-1}_{w_0})}(y)\leq \gamma \dot {(F^{-1}_{w_0})}(x)\leq \gamma R.
\]

For $k>1$ write:
\begin{eqnarray*}
\dot{(F^{-1}_{w_{[0,k)}})}(x)&=&
\dot{(F^{-1}_{w_0})}(x)\cdot
\dot{(F^{-1}_{w_1})}(x_1)
\cdots
\dot{(F^{-1}_{w_{k-1}}})(x_{k-1})\\
\dot{(F^{-1}_{w_{[0,k)}})}(y)&=&
\dot{(F^{-1}_{w_0})}(y)\cdot
\dot{(F^{-1}_{w_1})}(y_1)
\cdots
\dot{(F^{-1}_{w_{k-1}}})(y_{k-1}).
\end{eqnarray*}
We have $\rho(x,y)<\delta$. Let $I$ be the shorter closed interval between
$x$ and $y$. Then for $z\in I$ we have $\rho(x,z)<\delta$ and so
$\dot{(F^{-1}_{w_0})}(z)<1$. Thus $F^{-1}_{w_0}$ contracts $I$ and so
\[
\rho(x_1,y_1)=\rho(F^{-1}_{w_0}(x_0),F^{-1}_{w_0}(y_0))\leq \rho(x_0,y_0)<\delta.
\]
By the same argument, $F^{-1}_{w_1}$ contracts the shorter interval 
between $x_1$ and $y_1$ and so
$\rho(x_2,y_2)\leq\rho(x_1,y_1)<\delta$. Continuing in this manner, we obtain
that for all $i$ we have $\rho(x_i,y_i)<\delta$ and so
$\dot{(F^{-1}_{w_i})}(y_i)\leq\gamma\dot{F^{-1}_{w_i}}(x_i)$. Calculating the products,
we obtain:
\[
\dot{(F^{-1}_{w_{[0,k)}})}(y)\leq\gamma^k\dot{(F^{-1}_{w_{[0,k)}})}(x)\leq\gamma^kR^k=(\gamma
R)^k<1,
\]
concluding the proof.
\end{proof}

One might ask whether the condition (3) was necessary in
Theorem~\ref{thm-nonexistent3}. There exist trivial and less trivial examples
showing that (1) and (2) can indeed happen at once.

As a trivial example, consider the signed binary system from
Example~\ref{expbinary} with $W_i=\Te$ for $i\in\{0,1,\one,2\}$ and $\Sigma\subset
\{0,1,\one,2\}^\omega$ defined by forbidding words
$20,02,12,\one2,1\one,\one1$. It is easy to see that $\Sigma$ is compatible
with $\We$ and $\Sigma_\We=\Sigma$. As we have already shown, $\Sigma$ is a
Möbius number system. Observe that $F_0$ is hyperbolic with fixed points $\pm i$ and
$\dot{(F^{-1}_0)}(i)=\frac12$. Because $0^\omega\in\Sigma$, we have $Q_n(\We,\Sigma)\leq
\frac1{2^n}$ for each $n$ and so $Q(\We,\Sigma)\leq \frac12 <1$.

One might argue that we have cheated by taking $W_i=\Te$, that maybe some sort
of size limit on $W_i$ or allowing only shifts of the form $(\Aomega)_\We$ might prevent
(1) and (2) from happening simultaneously. While we are unable to account for
all the possible modifications of Theorem~\ref{thm-nonexistent3}, we show a less 
trivial number system that demonstrates the limitations of
Theorem~\ref{thm-nonexistent3}.

\begin{example}
Take the hyperbolic number system for $n=4$ and $r=\sqrt2-1$ from \cite{Kurka}. 
This iterative system consists of four hyperbolic
transformations as depicted in Figure~\ref{figslowwalk}. This system consists
of four hyperbolic MTs $\{F_0,F_1,F_2,F_3\}$ conjugated by rotation.
The transformation $F_1$ is a contraction to 1, $F_0$ has the stable fixed point $-i$
and unstable fixed point $i$, $F_2=F_0^{-1}$ and
$F_3=F_1^{-1}$. Moreover, the parameter of the contraction is chosen so that
\[
V_0=(e^{-3\pi/4i},e^{-\pi/4i}),\,
V_1=(e^{-\pi/4i},e^{\pi/4i}),\,
V_2=(e^{\pi/4i},e^{3\pi/4i}),\,
V_3=(e^{ 3\pi/4i},e^{-3\pi/4i}).
\]
As $\{V_i:i=0,1,2,3\}$ covers $\Te$, we obtain that the shift $\Omega$ from
Theorem~\ref{thmClosedCover2} is a Möbius number system. An argument similar to
the one used when analyzing Example~\ref{exp3parabolic} shows
that $\Omega$ is defined by the forbidden factors $02,20,13,31$.
\picture{The hyperbolic number system for $n=4$,
$r=\sqrt2-1$}{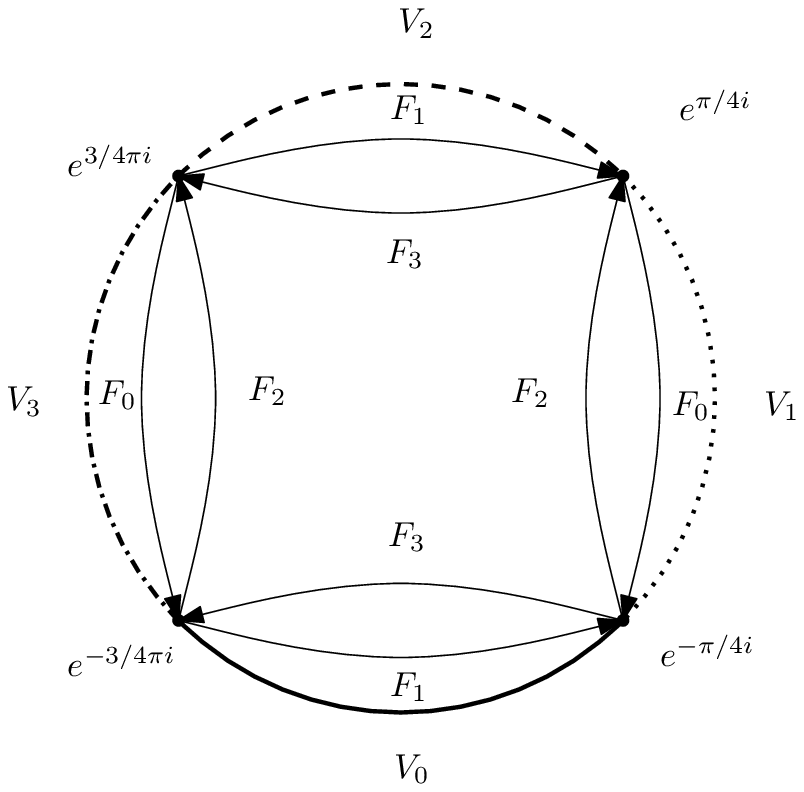}{figslowwalk}

However, we can obtain $\Omega$ as the interval shift of a rather different
interval almost cover. Take $W_i=\Te\setminus\oV_i$ for $i=0,1,2,3$. It turns out that
$\left(\Aomega\right)_\We=\Omega$. Moreover, $W_{0^n}$ always contains $-1$,
the unstable fixed point of $F_0$. Then $\dot{(F_0^{-1})}(-1)=q<1$ gives us 
$Q_n(\We,\Sigma)\leq q^n$ and so $Q(\We,\Sigma)\leq q<1$.
\end{example}

We have shown that if $Q(\We,\Sigma)>1$ then $\Sigma_\We$ is a Möbius number system, 
while if $Q(\We,\Sigma)<1$ then $\Sigma_\We$ might be a number system only if
it does not satisfy the rather reasonable condition $\Phi([u]\cap\Sigma_\We)=\oW_u$.

In the remaining case $Q(\We,\Sigma)=1$, the scales might tilt either way. 
We know that if $Q_n(\We,\Sigma)=1$ for some $n$ and no $F_v,\,v\in A^n$ is a rotation,
then $\Sigma_\We$ is a Möbius number system by Corollary~\ref{corClosedCover1}.
However, having $Q(\We,\Sigma)=1$ together with all the transformations $F_u$
different from rotations is not enough to obtain a Möbius number system:

\begin{example}
Take the three parabolic transformations from Example~\ref{exp3parabolic} and
let $W_a=(C,A)$, $W_b=(A,B)$, $W_c=(B,C)$ and $\Sigma=\Aomega$. 
It is easy to see that $W_{e^n}=W_e$ for every $e\in\{a,b,c\}$ while
$W_u=\emptyset$ whenever $u$ contains two different letters.
Therefore
$\Sigma_\We=\{a^\omega,b^\omega,c^\omega\}$, so this is not a Möbius
number system.

Now consider the transformation $F_{a^n}$. It is parabolic with the fixed point $a$
and $\dot{(F^{-1}_{a^n}(x))}$ is minimal for $x=C$. Denote
$f(n)=\dot{(F^{-1}_{a^n})}(C)$. Easily, $0<f(n)<1$. What is more,
$f(n)=|g(n)|^{-2}$ where $g$ is a linear polynomial in $n$ because $F_a$ is
parabolic (and therefore the matrix of $F_a$ is similar to
$\left(\begin{smallmatrix}1&\pm 1\\0&1\end{smallmatrix}\right)$).

It follows that $\sqrt[n]{f(n)}\to 1$ and the same argument holds for $b$ and $c$, so
we have $Q(\We,\Sigma)=1$.
\end{example}

As we have shown, the number $Q(\We,\Sigma)$, while useful, is not a perfect
solution for characterizing Möbius number systems. We will shortly mention one
direction (due to Petr Kůrka), in which it can be improved.

Let
$
Q(\Sigma)=\sup\{Q(\We,\Sigma):\We \hbox{ is an interval almost cover compatible with
$\Sigma$}\}.$

The number $Q(\Sigma)$ depends only on the iterative system and $\Sigma$.
Observe that $Q(\Sigma)>1$ means that there exists an interval almost cover $\We$ such that $\Sigma_\We$
is a Möbius number system while $Q(\Sigma)<1$ signifies that any possible
$\Sigma_\We$ is either not a Möbius number system at all or it is badly behaved
with respect to the condition $\Phi([u]\cap\Sigma_\We)=\oW_u$. It remains open
if $Q(\Sigma)$ has any other interesting properties. For example, it might be
possible that $Q(\Sigma)<1$ implies that no $\Sigma_\We$ is a Möbius system.

\subsection{Existence results}\label{secExistence}
The most basic existence question one might ask is whether there exists any
Möbius number system at all for a given iterative system. This problem is not
solved yet, however, we can offer a partial answer.

Originally, the following theorem comes from \cite{Kurka2}. We have slightly
modified it so that it refers to $\Te$ instead of the extended real line. It gives
one sufficient and one necessary condition for the existence of a Möbius
iterative system.

\begin{theorem}[Theorem 9, \cite{Kurka2}]\label{thmTheorem9}
Let $F:\Aplus\times\Te\to\Te$ be a Möbius iterative system.
\begin{enumerate}
\item If $\overline{\bigcup_{u\in\Aplus} V_u} \neq \Te$ then
$\Phi(\ex_F)\neq\Te$.
\item If $\{V_u:u\in \Aplus\}$ is a cover of $\Te$ then $\Phi(\ex_F)=\Te$ and
there exists a subshift $\Sigma\subset \ex_F$ on which $\Phi$ is continuous and
$\Phi(\Sigma)=\Te$.
\end{enumerate}
\end{theorem}

Note that if the condition of (1) is satisfied then there is no Möbius number
system for $\{F_a:a\in A\}$. We improve Theorem~\ref{thmTheorem9} by weakening
the condition in part (2).

Observe that (by compactness of $\Te$) if $\{V_u:u\in\Aplus\}$ cover $\Te$ then
there exists a finite $B\subset \Aplus$ such that $\{V_b:b\in B\}$ cover $\Te$. In the spirit
of our previous results, we show that in part (2) of Theorem~\ref{thmTheorem9} it suffices to demand that there exists a
finite $B\subset \Aplus$ such that the \emph{closed sets} $\{\oV_b:b\in B\}$ cover $\Te$.

\begin{corollary}\label{corClosedCover3}
Let $\{F_a:a\in A\}$ be a Möbius iterative system. Assume that there exists
a finite subset $B$ of $\Aplus$ such that $\{\oV_b:b\in B\}$ is a cover of
$\Te$. Then there exists a subshift $\Sigma \subset \Aomega$ that, together with
the iterative system $\{F_a:a\in A\}$, forms a Möbius number system.
\end{corollary}

\begin{proof}
Take $B$ as our new alphabet. 
The set $\{V_b:b\in B\}$ is an interval almost cover. Therefore, we can apply
Theorem~\ref{thmClosedCover2} and obtain the Möbius number
system $\Omega\subset\Bplus$.

We now proceed similarly to the proof of Corollary~\ref{corClosedCover2}:
Denote by $\psi$ the natural map from $\Bomega$ to $\Aomega$. Let $\Sigma=\bigcup_{i=0}^d
\sigma^d(\psi(\Omega))$ where $d=\max\{|b|:b\in B\}$. It is then straightforward to show that $\Sigma$ is a
Möbius number system for the iterative system $\{F_a:a\in A\}$.
\end{proof}

While Theorem~\ref{thmTheorem9} gives a necessary condition for a Möbius number
system to exist, this condition is not very comfortable to use. In the
spirit of \cite{Kurka-iterated}, we offer a condition that is easier to
check.

Let $\{F_a:a\in A\}$ be a Möbius iterative system.
A nonempty closed set $W\subset \Te$ is \emph{inward} if 
$\bigcup_{a\in A}F_a(W)\subset\Int(W)$. All iterative
systems have the trivial inward  set $\Te$ and some systems have nontrivial
inward sets as well.

\begin{theorem}\label{thminward}
Let $\{F_a:a\in A\}$ be an iterative system with a nontrivial inward set.
Then there is no Möbius number system for $\{F_a:a\in A\}$.
\end{theorem}
\begin{proof}
Assume that $W$ is a nontrivial inward set. Let $z\not\in W$. Because $W^c$ is
open there exists an open interval $I$ disjoint with $W$ and containing $z$. Assume
that $\Phi(w)=z$. Then $\lim_{k\to\infty}|F^{-1}_{w_{[0,k)}}(I)|=2\pi$.
However, $\Int W$ is nonempty so there exists an open interval $J\subset \Int
W$. Now for all $k$ we have $F_{w_{[0,k)}}(J)\subset W$ so
$J\cap F^{-1}_{w_{[0,k)}}(I)=\emptyset$. But then $|F^{-1}_{w_{[0,k)}}(I)|\leq
2\pi-|J|$, a contradiction.
\end{proof}

\subsection{Empirical data}\label{secData}
Searching for the solution to the existence problem, we have used a numerical
simulation to obtain insight in the behavior of MTs. The results suggest that
closing the gap in Theorem~\ref{thmTheorem9} is an achievable task.
 
We have studied the behavior of the iterative system $\{F_a,F_b\}$ consisting
of two hyperbolic transformations. The transformation $F_a$ has fixed points
$1$ (stable) and $-i$ (unstable), while the transformation $F_b$ has fixed
points $-1$ (stable) and $i$ (unstable). We have parameterized $F_a,F_b$ by
the values $q_a,q_b$ of $\dot{(F_i)}$ at stable points, so we have:
\begin{eqnarray*}
F_a&=&\frac{1}{2\sqrt{q_a}}
\begin{pmatrix}
1+q_a-i(1-q_a)&1-q_a+i(1-q_a)\\
1-q_a-i(1-q_a)&1+q_a+i(1-q_a)\\
\end{pmatrix}\\
F_b&=&\frac{1}{2\sqrt{q_b}}\begin{pmatrix}
1+q_b-i(1-q_b)&-1+q_b-i(1-q_b)\\
-1+q_b+i(1-q_b)&1+q_b+i(1-q_b)\\
\end{pmatrix}.
\end{eqnarray*}
Denote 
\[
Y=\{(q_a,q_b):q_a,q_b\in(0,1),\,\hbox{there exists a Möbius number system for $\{F_a,F_b\}$}\}.
\]

We wrote a C program that tries various pairs $(q_a,q_b)$, constructs
$F_a,F_b$, then computes the intervals $\{\oV_v:|v|\leq m\}$ (where $m$ is the
number of iterations to consider) and finally
checks whether these intervals cover the whole $\Te$. If they do, the program
puts a white
dot on the corresponding place in the graph, otherwise we leave it black. 

As we are interested in characterization, we have plotted (in gray) a second
set in the
graph: The set $U$ of all choices of $(q_a,q_b)$
such that the iterative system $\{F_a,F_b\}$ has a nontrivial inward set. The
formula for $U$, as shown in~\cite{Kurka-iterated}, is $U=\bigcup_{n\in\zet}
U_n$, where for $n>0$ we have:
\begin{eqnarray*}
U_0&=& U_{ab}\cap \left(0,\frac12\right)\times\left(0,\frac12\right)\\
U_n&=& U_{a^nb}\cap U_{a^{n+1}b}\cap
\left(\frac1{\sqrt[n]2},\frac1{\sqrt[n+1]2}\right)\times\left(0,\frac12\right)\\
U_{-n}&=&U_{ab^n}\cap U_{ab^{n+1}}\cap
\left(0,\frac12\right)\times\left(\frac1{\sqrt[n]2},\frac1{\sqrt[n+1]2}\right)
\end{eqnarray*}
and $U_v=\{(q_a,q_b):\hbox{$F_v$ is hyperbolic}\}$ for $v\in\Astar$. For practical reasons, we have
only drawn the sets $U_n$ with $|n|\leq m$ (the same $m$ as the number of
iterations in the first part of the program). The result for $m=10$ and
resolution $1000\times1000$ is shown in Figure~\ref{fig-depth}.

By Theorem~\ref{thminward}, $U\cap Y=\emptyset$. We are interested in the
size of the complement of $U\cup Y$ in $(0,1)^2$. It turns out that $U\cup Y$
covers most of the unit square and $U$ and $Y$ appear to fit rather well together. 
The area between the two sets is likely
to get smaller and smaller as we let $m$ grow,
possibly shrinking to zero in the limit. However, there might still exist points 
that do not belong either to $U$ or to $Y$. 
We (vaguely) conjecture, that $U\cup Y$ is equal to the whole $(0,1)^2$
perhaps up to a small set of exceptional points.
\begin{figure}
\begin{center}
\includegraphics[width=\textwidth]{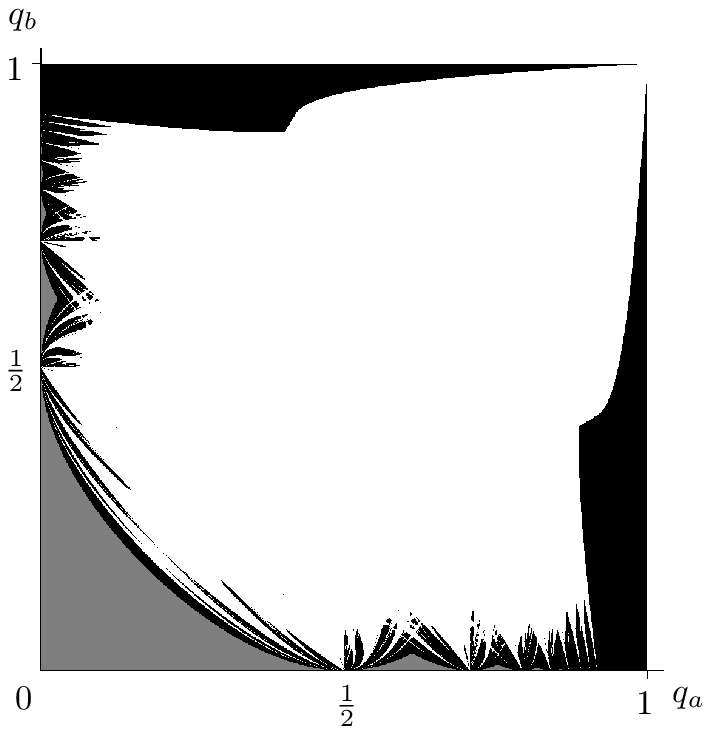}
\caption{The graph of $U$ and $Y$ for the depth $m=10$ and resolution
$1000\times1000$.}
\label{fig-depth}
\end{center}
\end{figure}

\subsection{Subshifts not admitting Möbius number systems}\label{secSubshifts}
An interesting question one can pose is whether, given a subshift
$\Sigma\subset\Aomega$, there exists a collection of MTs $\{F_a:a\in A\}$ such
that $\Sigma$ is a Möbius number system.

Trivially, the cardinality of $\Sigma$ must be precisely continuum, as for
smaller $\Sigma$ there is no projection from $\Sigma$ onto $\Te$. We offer a
less trivial necessary condition.

A (non-erasing) \emph{substitution} is any mapping $\psi:B\to \Aplus$. We
can extend $\psi$ to a map from $\Bomega$ to $\Aomega$ in a natural way. As
there is no risk of confusion, we will denote the resulting map by $\psi$ as
well and call it a \emph{substitution map}. Observe that
$\psi:\Bomega\to\Aomega$ is continuous in the product topology.
We have already met substitution maps in the proof of Corollary~\ref{corClosedCover2}.
 
\begin{theorem}\label{thm-nofull}
Let $\Sigma\subset \Aomega$ be a Möbius
number system for an iterative system $\{F_a:a\in A\}$.
Then for all alphabets $B$ and all substitution maps $\psi$ we have $\Sigma\neq \psi(\Bomega)$.
\end{theorem}

\begin{proof}
Assume that there exist $B$ and $\psi$ for which the theorem is false.

Then $\Bomega$ together with the maps $\{G_b:b\in B\}$ such that
$G_b=F_{\psi(b)}$ is a Möbius number system. Denote by $\Phi$ the resulting
projection of $\Bomega$ to $\Te$ and observe that $\Phi=\Phi_{\Sigma}\circ\psi$
where $\Phi_{\Sigma}:\Sigma\to\Te$ is the number system on $\Sigma$. We see
that $\Phi$ is surjective and continuous on $\Bomega$.

Fix $x\in \Te$ and $u\in\Bomega$. There exists $v\in\Bomega$ with $\Phi(v)=x$.
Consider the sequence $\{\Phi(u_{[0,k)}v)\}_{k=1}^\infty$. We have
\[
\Phi(u_{[0,k)}v)=G_{u_{[0,k)}}(\Phi(v))=G_{u_{[0,k)}}(x).
\]
However, $\Phi$ is
continuous, so $\Phi(u_{[0,k)}v)$ tends to $\Phi(u)$ when $k$ tends to
infinity. Therefore, $G_{u_{[0,k)}}(x)\to \Phi(u)$.

As $x,u$ were arbitrary, we have shown that for every $u\in\Bomega$ and each
point $x$ of $\Te$ the sequence$\{G_{u_{[0,k)}}(x)\}_{k=1}^\infty$ converges to
$\Phi(u)$.

Let $u$ be periodic with some period $w\in\Bplus$. Then $\Phi(u)$ is the
(stable) fixed
point of $G_w$. Therefore $G_w$ may not be elliptic ($\Phi(u)$ would not be
defined) nor hyperbolic (the images of the unstable fixed point of $G_w$
would not converge). This means that for all $w\in\Bplus$ the 
transformation $G_w$ must be parabolic.

If all the the transformations $G_b,\, b\in B$ were parabolic with the same
fixed point then all the transformations $G_w, w\in B^+$ would be parabolic
with the same fixed point and we would not be able to represent anything except
this fixed point. Therefore there exist $a,b\in B$ such that $G_a,G_b$ have
different fixed points. The rest of the proof consists of a
straightforward (but technical) calculation that such a situation is impossible.

Choose $a,b\in B$ so that $G_a,G_b$ have different fixed points. We know that $G_{ab}$
and $G_{aab}$ must both be parabolic. Without loss of generality assume that
$G_a$ is similar to the matrix
$J=\begin{pmatrix}1&1\\0&1\end{pmatrix}$. Write $G_a=M J M^{-1}$
where $M$ is a regular matrix. Recall that a transformation $F$ is parabolic
iff $\Tr(F)^2=4$ and observe that:
\[\Tr(G_{ab})=\Tr(MJM^{-1}G_b)=\Tr(JM^{-1}G_bM)\]
and
\[\Tr(G_{aab})=\Tr(MJ^2M^{-1}G_b)=\Tr(J^2M^{-1}G_bM),\]
where we used the equality $\Tr(AB)=\Tr(BA)$.
Let 
\[
M^{-1}G_bM=\begin{pmatrix}
a&b\\
c&d\\
\end{pmatrix}
\]
and observe that $(a+d)^2=\Tr(G_b)^2=4$.

We now have
\[
JM^{-1}G_bM=\begin{pmatrix}
a+c&b+d\\
c&d\\
\end{pmatrix},\quad
J^2M^{-1}G_bM=\begin{pmatrix}
a+2c&b+2d\\
c&d\\
\end{pmatrix}.
\]
Calculating the traces, we obtain the equalities 
\[
(a+2c+d)^2=(a+c+d)^2=(a+d)^2=4
\]
that can only
be satisfied when $c=0$. But then matrices $J$ and $M^{-1}G_bM$ share the
eigenvector $(1,0)^T$ and so $G_a=M J M^{-1}$ and $G_b$ share the eigenvector
$M(1,0)^T$. However, eigenvectors of matrices are in one to one correspondence
with fixed points of Möbius transformations (see the proof of Lemma~\ref{lemFixedpoints} in
the Appendix) and so $G_a,G_b$ have the same fixed point, a contradiction.
\end{proof}

Theorem~\ref{thm-nofull} tells us in particular that there are no Möbius
number systems on the full shift. We must always take something smaller in
order to limit bad concatenations. An important weakness of
Theorem~\ref{thm-nofull} is that the set $\psi(\Bomega)$ \emph{need not be a
subshift}: while it is always closed, $\sigma$-invariance is not guaranteed.

However, there are cases when $\psi(\Bomega)$ is a nontrivial subshift. Consider
the Fibonacci shift $\Sigma_F$ defined on the alphabet $\{0,1\}$ by forbidding
the factor $11$. It is easy to see that $\Sigma_F=\psi(\{0,1\}^\omega)$ under
the substitution $\psi:0\mapsto 0, 1\mapsto10$. Therefore, $\Sigma_F$ can never
be a Möbius number system.

\subsection{Sofic Möbius number systems}\label{secSofic}
In this subsection, we will explore another facet of Möbius number systems. As
every number system is a subshift, we can ask how complicated (in the sense of
formal language theory, not information theory) is the language of this subshift.

A subshift $\Sigma$ is of \emph{finite type} if $\Sigma$ can be defined using a
finite set of forbidden words (note that this was the case in all our example
subshifts). A subshift $\Sigma$ is called \emph{sofic} if and only if the language
of $\Sigma$ is regular (recognizable by a finite automaton). A little thought
gives us that subshifts of finite type are always sofic. Sofic subshifts and
subshifts of finite type are quite popular in practice, as they are easier to
manipulate than general subshifts. There are numerous
results and algorithms available for sofic subshifts and subshifts of finite type.

The papers \cite{Kurka} and
\cite{Kurka2} contain several examples of Möbius number systems that are
subshifts of finite type. Furthermore, Proposition 5 in~\cite{Kurka2} states a
sufficient condition for a number system to be of finite type. We now 
present a similar condition for $\Sigma_\We$ to be sofic.

\begin{theorem}\label{thmSofic}
Let $\Sigma$ be a sofic subshift and $\We$ such an interval almost cover that
the set $\{F^{-1}_v(W_v):v\in\Astar\}$ is
finite. Then $\Sigma_\We$ is sofic.
\end{theorem}
\begin{proof}
We construct a finite automaton $\mathcal A$ that recognizes all the words $v\in\Astar$ such
that $W_v\neq \emptyset$. We then intersect the resulting regular language with the
language  $\el(\Sigma)$ to obtain $\el(\Sigma_\We)$. Because regular
languages are closed under intersection, $\el(\Sigma_\We)$ is regular.

The states of our automaton will be all the sets $Z_v=F^{-1}_v(W_v),\,v\in
\Astar$. We let $Z_\lambda$ to be the initial state and all states except 
$\emptyset$ to be accepting states. A transition labelled by the letter $a$
leads from $Z_v$ to $Z_{va}$ for every $v\in\Astar$ and every $a\in A$. 

Observe that when $Z_v=Z_u$ then $Z_{va}=Z_{ua}$, so the definition of our
automaton is correct: 
\[
Z_{va}=F^{-1}_aF^{-1}_{v}(W_v\cap F_v(W_a))=F^{-1}_a(F^{-1}_{v}(W_v))\cap
F^{-1}_a(W_a)=F^{-1}_a(Z_v)\cap Z_a
\]
Similarly, $Z_{ua}=F^{-1}_a(Z_u)\cap Z_a$ and as $Z_u=Z_v$ we obtain
$Z_{va}=Z_{ua}$.

To finish the proof, we observe that the automaton $\mathcal A$
accepts the word $v$ iff $Z_v\neq \emptyset$. Because $Z_v\neq\emptyset$ iff
$W_v\neq\emptyset$, $\mathcal A$ recognizes precisely those $v\in\Astar$ with
$W_v\neq\emptyset$.
\end{proof}

Under an additional assumption, we can prove the converse of
Theorem~\ref{thmSofic}:  

\begin{theorem}\label{thmSofic2}
Assume that $\We$ is an interval almost cover compatible with the subshift
$\Sigma$. Let the subshift $\Sigma_\We$ be a sofic Möbius number system such
that $\Phi([v]\cap\Sigma_\We)=\oW_v$ for every word $v$.
Then the set $\{F^{-1}_v(W_v):v\in\Astar\}$ is finite.
\end{theorem}
\begin{proof}
To prove this theorem we define a chain of several 
finite sets, each obtained from the previous, with the final set being $\{F^{-1}_v(W_v):v\in\Astar\}$.

Denote by ${\mathcal F}(v)$ the \emph{follower set of $v$ in $\Sigma_\We$}, i.e. the
set of all words $u\in\Astar$ such that $vu\in\el(\Sigma_\We)$. By the Myhill-Nerode
theorem, we have that if $\Sigma_\We$ is sofic, then $\{{\mathcal
F}(v):v\in\Astar\}$ is a finite set. Let $v\in\Astar$ and denote ${\mathcal
F}^\omega(v)=\{w\in\Aomega:\forall k, w_{[0,k)}\in{\mathcal F(v)}\}$.
The set $\{{\mathcal F}^\omega(v):v\in\Astar\}$ is finite as each
${\mathcal F}^\omega(v)$ depends only on ${\mathcal F}(v)$.

A little thought gives us that ${\mathcal
F}^\omega(v)=\{w\in\Aomega:vw\in\Sigma_\We\}$. Finally, denote
$Z_v=\Phi({\mathcal F}^\omega(v))$ and observe again that the set $\{Z_v:v\in\Astar\}$ is
finite. It remains to notice that
\[
Z_v=\Phi(\{w\in\Aomega:vw\in\Sigma_\We\})=F^{-1}_v(\Phi([v]\cap\Sigma_\We))=F^{-1}_v(W_v)
\]
to see that the set $\{F^{-1}_v(W_v):v\in\Astar\}$ must be finite.
\end{proof}

Note that Theorems~\ref{thmSofic} and \ref{thmSofic2} give us that if $\Sigma$
is sofic then the interval shifts considered in Theorem~\ref{thmClosedCover1}
or Corollary~\ref{corClosedCover2} are sofic if and only if
$\{F^{-1}_v(W_v):v\in \Astar\}$ is a finite set.

\section*{Conclusions and open problems} 
\addcontentsline{toc}{section}{Conclusions}
In the whole paper we have explored various topics in the theory of Möbius
number systems. We have obtained tools to prove that a subshift is a Möbius
number system for a given iterative system as well as various existence results
and criteria for sofic number systems. However, we have left quite a few open
problems, practical as well as theoretical, in this area.

The first open question is how ``nice'' the representation of points is for
sequences of the type $\{F_{w_{[0,n)}}\}_{n=1}^\infty$. In general, we can have
a sequence of Möbius transformations $\{F_n\}_{n=1}^\infty$ that represents
$x\in\Te$, yet for
all $z\in\Te$ the sequence $\{F_n(z)\}_{n=1}^\infty$ does not converge to $x$. We think that
such sequences must always exhibit low speed of convergence of
$\{F_n(0)\}_{n=1}^\infty$ to $x$. We further conjecture that in the more
special case of
$w\in \Aomega$ and $\{F_{w_{[0,n)}}\}_{n=1}^\infty$ representing some $x$, the
set of points $z\in\Te$ such that $F_{w_{[0,n)}}(z)$ converges to $x$ is nonempty, perhaps
even of measure one.

Interval shifts together with computing $Q(\We,\Sigma)$ and $Q(\Sigma)$ seem to
be useful when dealing with concrete examples. Theorem~\ref{thmClosedCover1}
and Corollary~\ref{thmClosedCover2} offer practical tools to prove that a given
subshift is a Möbius number system for a given iterative system. We wonder how
the available toolbox for this kind of proof could be further improved. We see
numerous areas open to incremental improvements.

For examples and applications, we would like to have a sufficient and necessary
condition for the existence of a Möbius number system for a given iterative
system. Ideally, this condition should be effectively verifiable (for
reasonable iterative systems, say when real and imaginary parts of coefficients
are rational). While we doubt that a general effective algorithm exists,
improvements in the tools for proving the existence of Möbius iterative systems
would be welcome indeed. 

Regarding iterative systems, we know that $Q(\We,\Sigma)>1$ or $Q(\Sigma)>1$
guarantees the existence of a number system while $Q(\We,\Sigma)<1$ or
$Q(\Sigma)<1$ means that if a number system $\Sigma_\We$ exists at all then it
is going to be badly behaved with respect to $\We$. We conjecture that when
$Q(\Sigma)<1$ there actually does not exist any number system on any interval
shift $\Sigma_\We\subset\Sigma$.

Again, we do not know if there exist general algorithms for computing $Q(\We,\Sigma)$
and $Q(\Sigma)$ even when $\Sigma$ is of finite type or sofic. Perhaps we
could gain some inspiration in the symbolic dynamic tools for computing entropy 
of shifts of finite type. Being able to compute, or at least estimate
$Q(\We,\Sigma)$ would make examining examples easier.

Another, perhaps less practical, but combinatorially interesting problem is
when a given subshift $\Sigma$ can be a Möbius number system. So far, we have
some sufficient and some necessary conditions and a large gap in between.

To manipulate number systems, it would be nice to have a sofic Möbius number
system. Theorems~\ref{thmSofic} and \ref{thmSofic2} offer useful checks to
perform when verifying if a number system is sofic. What is completely missing
is a condition, similar to Theorem~\ref{thmClosedCover1}, for the existence of a sofic number system for a given iterative
set. We are hoping that obtaining such a result is possible but the current
amount of knowledge on sofic number systems is rather small. For example, we
can not even tell whether existence of a Möbius number system implies
existence of a sofic system for the same iterative system or not. 

A large part the complexity of above problems seems to come not from the number
systems themselves but from the fact that we don't properly understand how do
large numbers of MTs compose (or, equivalently, how long sequences of matrices
multiply). This suggests that maybe the way forward lies in studying the limits
of products of matrices. Unfortunately, this area is full of hard questions,
see for example \cite{unsolvedblondel}.

Hard problems notwithstanding , we conclude on a positive note: Although there
are numerous open questions about Möbius number systems, and some properties of
these systems might turn out to be undecidable, current tools do allow us to
deal with systems that are likely to be used elsewhere (for example, the
continued fraction number system).

\section*{Acknowledgements}
\addcontentsline{toc}{section}{Acknowledgements}
The author thanks his supervisor Petr K\r urka for support and guidance.
The research was supported by the Czech Science Foundation research project 
GA\v{C}R 201/09/0854.
This text contains parts from author's previously published paper
\cite{Kazda}.

\section*{Appendix}
\addcontentsline{toc}{section}{Appendix}

The Appendix contains various proofs that we felt should be included in this
paper, yet their length or technical nature would disturb the flow of the rest
of the text. Note that these are all well known results; the proofs here are
just for the sake of completeness and understanding of the topic.

First, we present a series of three lemmas concerning the
classification of disc preserving MTs into elliptic, parabolic and hyperbolic
transformations.  We will use a bit of linear algebra machinery. We will
understand each disc preserving Möbius transformation $F$ both as a map and as the
corresponding normalized matrix
\[
F=\begin{pmatrix}
\alpha&\beta\\
\cc\beta&\cc\alpha\\
\end{pmatrix},\, |\alpha|^2-|\beta|^2=1.
\]

An important role in the following proofs belongs to the eigenvalues of the
matrix $F$. However, these eigenvalues are not
uniquely defined: The Möbius transformation $F$ always has two corresponding
normalized matrices $F,-F$, therefore it also has two different sets of
eigenvalues. We deal with this problem by always fixing one matrix of $F$ for
the whole proof.

\begin{lemma}\label{lemFixedpoints}
Let $F$ be a disc preserving Möbius transformation. Fix a matrix of $F$
such that $\det F=1$. Then the following holds:
\begin{enumerate}
\item $F$ is elliptic iff the eigenvalues of $F$ are not real iff
$F$ has one fixed point inside and one fixed point
outside of $\Te$ (the outside point might be $\infty$), 
\item $F$ is parabolic iff $F$ has the single eigenvalue equal to $1$ or $-1$ iff 
$F$ has a single fixed point and it lies on $\Te$,
\item $F$ is hyperbolic iff $F$ has two different real eigenvalues iff
$F$ has two different fixed points, both lying on $\Te$.
\end{enumerate}
\end{lemma}
\begin{proof}
We begin by providing a connection between $(\Tr F)^2$ and the eigenvalues of $F$.
The characteristic polynomial of $F$ is:
\[
(\alpha-\lambda)(\cc\alpha-\lambda)-\beta\cc\beta=|\alpha|^2-|\beta|^2-\Tr
F\cdot\lambda+\lambda^2=1-\Tr F\cdot \lambda+\lambda^2.
\]
We see that for $(\Tr F)^2<4$, $F$ has two distinct complex conjugate
eigenvalues, while if $(\Tr F)^2>4$, then $F$ has two distinct real
eigenvalues. Finally, if $(\Tr F)^2=4$, there is only one eigenvalue 
$\lambda=1$ or $\lambda=-1$. 

Observe that for all $z\in\ce$ such that $F(z)\neq \infty$ we have:
\[
F\cdot 
\begin{pmatrix} 
z \\
1 \\
\end{pmatrix}
=
\begin{pmatrix} 
\alpha z+\beta \\
\cc\beta z+\cc\alpha\\
\end{pmatrix}
=
(\cc\beta z+\cc\alpha)\cdot
\begin{pmatrix} 
F(z) \\
1 \\
\end{pmatrix}.
\]
For $z\in\ce$ we obtain that $z$ is a fixed point of $F$ iff $(z,1)^T$ is an
eigenvector of $F$. Similarly, the point $\infty$ is fixed iff $(1,0)^T$ is an
eigenvector of $F$.

Let $\lambda$ be an eigenvalue of $F$. If $F$ is not the identity, then the
eigenspace of $\lambda$ must have dimension $1$ (otherwise all the points of
$\Ce$ would be fixed points of $F$). Therefore, we have a one to one
correspondence between the eigenvalues and fixed points of $F$.

Let $v=(v_1,v_2)^T$ be an eigenvector corresponding to the eigenvalue $\lambda$. Then:
\begin{eqnarray*}
F\cdot
\begin{pmatrix} 
v_1 \\
v_2 \\
\end{pmatrix}
&=&
\begin{pmatrix} 
\alpha v_1+\beta v_2 \\
\cc\beta v_1+\cc\alpha v_2\\
\end{pmatrix}
=
\begin{pmatrix} 
\lambda
v_1 \\
\lambda
v_2 \\
\end{pmatrix}
=\lambda
\begin{pmatrix} 
v_1 \\
v_2 \\
\end{pmatrix}
\\
F\cdot
\begin{pmatrix} 
\cc v_2 \\
\cc v_1 \\
\end{pmatrix}
&=&
\begin{pmatrix} 
\beta \cc v_1 + \alpha \cc v_2 \\
\cc\alpha \cc v_1+ \cc\beta \cc v_2\\
\end{pmatrix}
=
\begin{pmatrix} 
\cc{\cc\beta v_1 +\cc \alpha v_2} \\
\cc{\alpha  v_1+ \beta  v_2}\\
\end{pmatrix}
=
\begin{pmatrix} 
\cc{\lambda v_2} \\
\cc{\lambda v_1} \\
\end{pmatrix}
=
\cc\lambda
\begin{pmatrix} 
\cc v_2 \\
\cc v_1 \\
\end{pmatrix}.\\
\end{eqnarray*}
If $\lambda$ is real, then the vectors $(v_1,v_2)^T$ and $(\cc v_2,\cc v_1)^T$
must be linearly dependent as they both belong to the same eigenspace. In
particular, if $v_2=0$ we would have $v_1=0$, so we can assume that $v_1=z,
v_2=1$. But then the linear dependence is equivalent with
\[
\det
\begin{pmatrix} 
z&1\\
1&\cc z\\
\end{pmatrix}=0
\] which is equivalent with $|z|=1$ and so $z\in \Te$.

On the
other hand, if $\lambda$ is not real then $\cc\lambda\neq\lambda$ and the
vectors $(v_1,v_2)^T$ and $(\cc v_2,\cc v_1)^T$ must be linearly independent as
they are eigenvectors of different eigenvalues. Assume $z\in\ce$ is a fixed
point of $F$.  Were $|z|=1$ then the determinant argument above would give us
a contradiction with linear independence. Therefore $z\not\in \Te$. But there
is more: If $z$ is a fixed point of $F$ then so is $\frac1{\cc z}$, the image
of $z$ under circle inversion with respect to $\Te$:
\[
F\cdot 
\begin{pmatrix}
\frac 1{\cc z}\\
1
\end{pmatrix}
=
\frac1{\cc z}
F\cdot 
\begin{pmatrix}
 1\\
\cc z\\
\end{pmatrix}
=
\frac1{\cc z}
\cc\lambda
\begin{pmatrix}
 1\\
\cc z
\end{pmatrix}
=
\cc\lambda
\begin{pmatrix}
\frac 1{\cc z}\\
1
\end{pmatrix}.
\]
Similarly, if $\infty$ is a fixed point then so is $0$. Therefore, if $F$ has
an eigenvalue that is not real, then $F$ has one fixed point outside $\Te$ and
one inside $\Te$ (and we can even map one onto another using the circle
inversion with respect to $\Te$). This is precisely the case of elliptic $F$.

On the other hand, if $F$ is hyperbolic, then $F$ has two distinct real
eigenvalues and, therefore, two distinct fixed points on $\Te$.

If $F$ is elliptic then there is a single real eigenvalue of $F$ and so the Möbius
transformation $F$ has a single fixed point on $\Te$.
\end{proof}

\begin{lemma}\label{lemFixedPoints1}
Assume $F$ is a hyperbolic transformation, $x_1$ and $x_2$ its
fixed points. Then $F'(x_1)=\lambda_2/\lambda_1$ and 
$F'(x_2)=\lambda_1/\lambda_2$ where $\lambda_1,\lambda_2$ are the
eigenvalues of $F$ associated to $x_1$ and $x_2$.

Similarly, if $F$ is a parabolic transformation and $x$ its fixed point
then $F'(x)=1$.
\end{lemma}
\begin{proof}
First observe that the ratio of $\lambda_1$ and
$\lambda_2$ does not depend on the choice of the matrix for $F$ so the claim is
sensible.

We will fix a normalized matrix corresponding to $F$.
Let $J$ be the Jordan matrix 
similar to the matrix $F$, i.e. then exists an MT $M$ such
that $F=M\circ J\circ M^{-1}$.

Let $x$ be a fixed point of $F$. Then we have:
\[
F'(x)=(M\circ J\circ M^{-1})'(x)=M'(J\circ M^{-1}(x))\cdot J'(M^{-1}(x))\cdot
(M^{-1})'(x).
\]
Because $F(x)=x$, we must have $J\circ M^{-1}(x)=M^{-1}(x)$, so:
\[
F'(x)=M'(M^{-1}(x))\cdot J'(M^{-1}(x))\cdot (M^{-1})'(x)=J'(M^{-1}(x)),
\]
where we used the formula $M'(M^{-1}(x))\cdot (M^{-1})'(x)=1$.

However, $J'(M^{-1}(x))$ is easy to compute because $M^{-1}(x)$ is the fixed
point of $J$ corresponding to the same eigenvalue as $x$. The only problem is that
we have to ensure $M^{-1}(x)\neq \infty$ to have the derivative well defined.
We deal with this problem by carefully choosing our $J$.

Let us begin with the hyperbolic case. To calculate, $x_1$ choose the Jordan
matrix $J=\left(\begin{smallmatrix}
\lambda_2&0\\
0&\lambda_1\\
\end{smallmatrix}\right)$ (note the switched order of $\lambda_1,\lambda_2$).
Now $J(z)=\frac{\lambda_2}{\lambda_1}z$ and $M^{-1}(x_1)=0$. Thus
$F'(x_1)=J'(0)=\frac{\lambda_2}{\lambda_1}$. Similarly, we calculate
$F'(x_2)=\frac{\lambda_1}{\lambda_2}$ from
the Jordan form $J=\left(\begin{smallmatrix}
\lambda_1&0\\
0&\lambda_2\\
\end{smallmatrix}\right)$.

In the parabolic case, we avoid problems with the point at $\infty$, by taking
a matrix similar to the usual Jordan form but with different interpretation as
a Möbius transformation:
$J=\left(\begin{smallmatrix}
1&0\\
\mp1&1\\
\end{smallmatrix}\right)$ (the sign in the lower left corner depends on the
Jordan matrix for $F$; there are two possibilities). Now
$J(z)=\frac{z}{\mp z+1}$ and $M^{-1}(x)=0$, so $F'(x)=J'(0)=1$ and we
are done.
\end{proof}

\begin{lemma}\label{lemFixedPoints2}
Let $F$ be a parabolic or a hyperbolic transformation, let $x$ be the (stable) fixed
point of $F$. Let $z\in\Ce$ (assume that $z$ is not the unstable fixed point of $F$ in
the hyperbolic case). Then $\lim_{n\to\infty}F^n(z)=x$.
\end{lemma}
\begin{proof}
Let us again fix a matrix of $F$ such that $\det F=1$.
Denote by $J$ the Jordan matrix similar to $F$. 
Then $F=M\cdot J \cdot M^{-1}$ and $F^n=M\cdot J^n\cdot M^{-1}$. We know that
$J^n$ has one of the two possible forms:
\[
\begin{pmatrix}
1&\pm n\\
0&1\\
\end{pmatrix}, \hbox{ or}\;
\begin{pmatrix}
\lambda_1^n&0\\
0&\lambda_2^n\\
\end{pmatrix}
\]
where without loss of generality $|\lambda_1|>1>|\lambda_2|$.
Let $z\in\Ce$ (if $F$ is hyperbolic, let $z$ be different from the unstable
fixed point of $F$)
and consider the vector
\[
F^n\begin{pmatrix}z\\1\end{pmatrix}=M^{-1}\cdot J^n\cdot
M\cdot \begin{pmatrix}z\\1\end{pmatrix}.
\]

It is easy to see that the first component of $J^n\cdot M\cdot \left(\begin{smallmatrix}z\\1\end{smallmatrix}\right)$ tends to infinity, while the
second is bounded. Therefore (understanding $J$ and $M$ as Möbius
transformations), we have $\lim_{n\to\infty} J^n\circ M(z)=\infty$.

However, $M^{-1}(\infty)=x$ as $(x,1)^T$ and $(1,0)^T$ are eigenvectors of $F$
and $J$ belonging to the same eigenvalues. It follows that  
\[
\lim_{n\to\infty} M^{-1}\circ J^n\circ M(z)=M^{-1}(\infty)=x.\qedhere
\]
\end{proof}

We conclude the Appendix with two miscellaneous lemmas: a lemma proving
continuity of a certain linear functional and Fekete's lemma about
superadditive series.

The following lemma is a consequence of the Riesz representation theorem. See
(\cite[page 184]{ash} for details). In its statement, we are going to identify
$C^*(\Te,\er)$ with the space of signed Radon measures on $\Te$.
\begin{lemma}\label{lemContinuity}
Let $E$ be a measurable set on $\Te$. Then the map
$\alpha:\lambda\mapsto \lambda(E)$ from $C^*(\Te,\er)$ to $\er$ is
linear and continuous on $C^*(\Te,\er)$.
\end{lemma}
\begin{proof}
Linearity of $\alpha$ is obvious. To obtain continuity, it is enough to show
that $|\alpha(\lambda)|$ is bounded whenever $|\lambda|$ is bounded. By
definition,
\[
|\lambda|=\sup\{\lambda(f):f\in C(\Te,\er),|f|\leq 1\}.
\]

We first observe that
every $\lambda$ can be written as $\lambda_1-\lambda_2$ where
$\lambda_1,\lambda_2$ are positive measures. Moreover, as shown in \cite{ash},
we can choose $\lambda_1, \lambda_2$ so that 
\begin{eqnarray*}
\lambda_1(f)&=&\sup\{\lambda(g):0\leq g\leq f\}\\
\lambda_2(f)&=&\sup\{\lambda(g):-f\leq g\leq 0\}
\end{eqnarray*}
for all $f\geq 0$.

Then we have
\[
|\lambda|\geq \sup\{\lambda(f):\, 0 \leq f\leq
1\}=\lambda_1(1)=\lambda_1(\Te)\geq \lambda_1(E)\geq \lambda(E)
\]
as well as
\[
|\lambda|\geq \sup\{\lambda(f):\, -1 \leq f\leq0\}
=\lambda_2(1)=\lambda_2(\Te)\geq \lambda_2(E)\geq -\lambda(E).
\]
Therefore, $|\lambda|\geq |\lambda(E)|=|\alpha(\lambda)|$ for every 
$\lambda$, proving the continuity of $\alpha$.
\end{proof}

\begin{lemma}[Fekete's lemma]\label{lemSupadditive}
Let $\{R_n\}_{n=1}^\infty$ be a \emph{superadditive sequence} of real numbers, that is, a
sequence such that for all $m,n\in\en$ we have $R_{m+n}\geq R_m+R_n$. Then the limit
$\lim_{n\to\infty}\frac1n R_n$ exists and is equal to $\sup_{n\in\en} \frac1n
R_n$.
\end{lemma}
\begin{proof}
Denote $s=\sup_{n\in\en} \frac1n R_n$. 

Assume first $s<\infty$. Let $\epsilon>0$ and suppose that $n$ is such that $\frac1n R_n>s-\epsilon$. From the
superaditivity condition, we obtain for any $k\in \en$ and any $m<n$ the
inequalities: 
\begin{eqnarray*}
\frac1{kn+m}R_{kn+m}&\geq&\frac1{kn+m}\left( k R_n+mR_1\right)\\
&\geq& \frac{kn(s-\epsilon)+mR_1}{kn+m}\\
&\geq& \left(1-\frac1{k+1}\right)(s-\epsilon)+\frac{m}{kn+m}R_1
\end{eqnarray*}
In particular, there exists $k$ such that whenever $l>kn$, the value
$\frac{1}{l}R_l$ belongs to the interval $[s-2\epsilon,s]$. Therefore, the
sequence $\{\frac1n R_n\}_{n=1}^\infty$ converges to $s$.

If $s=\infty$, replacing $s-\epsilon$ with arbitrarily large $K>0$ and
performing the same argument gives us that there exist $n$ and $k$ such that
$l>kn$ implies $\frac1l R_l\geq K-1$, proving the lemma.
\end{proof}

\addcontentsline{toc}{section}{Bibliography}
\bibliography{citations}

\end{document}